\documentclass[11pt]{amsart}%
\usepackage[dvips]{graphicx,mathtools}
\usepackage{amsmath}
\usepackage{color}
\usepackage{amsfonts}
\usepackage{amssymb}%
\providecommand{\U}[1]{\protect\rule{.1in}{.1in}}
\providecommand{\U}[1]{\protect\rule{.1in}{.1in}}
\providecommand{\U}[1]{\protect\rule{.1in}{.1in}}
\providecommand{\U}[1]{\protect\rule{.1in}{.1in}}
\providecommand{\U}[1]{\protect\rule{.1in}{.1in}}
\providecommand{\U}[1]{\protect\rule{.1in}{.1in}}
\providecommand{\U}[1]{\protect\rule{.1in}{.1in}}
\newtheorem{theorem}{Theorem}[section]

\newtheorem{algorithm}[theorem]{Algorithm}

\newtheorem{corollary}[theorem]{Corollary}

\newtheorem{definition}[theorem]{Definition}
\newtheorem{example}[theorem]{Example}

\newtheorem{lemma}[theorem]{Lemma}

\newtheorem{proposition}[theorem]{Proposition}
\newtheorem{remark}[theorem]{Remark}
\newtheorem{remarks}[theorem]{Remarks}

\newcommand{\casetwoother}[3]{\left\{ \begin{array}{ll} #1 &\mbox{if $#2$} \\[1mm] #3 &\mbox{otherwise}\,. \end{array} \right.}
\newcommand{\casethree}[6]{\left\{ \begin{array}{ll} #1 &\mbox{if $#2$} \\[1mm] #3 &\mbox{if $#4$} \\[1mm]#5 &\mbox{if $#6$}\,. \end{array} \right.}

\newlength{\cellsize}
\newcommand\tableau[1]{
\vcenter{
\let\\=\cr
\baselineskip=-16000pt
\lineskiplimit=16000pt
\lineskip=0pt
\halign{&\tableaucell{##}\cr#1\crcr}}}

\newcommand{\tableaucell}[1]{{%
\def \arg{#1}\def \void{}%

\ifx \void \arg
\vbox to \cellsize{\vfil \hrule width \cellsize height 0pt}%
\else
\unitlength=\cellsize
\begin{picture}(1,1)
\put(0,0){\makebox(1,1){$#1$}}
\put(0,0){\line(1,0){1}}
\put(0,1){\line(1,0){1}}
\put(0,0){\line(0,1){1}}
\put(1,0){\line(0,1){1}}
\end{picture}%
\fi}}

\begin{document}
\title[Combinatorics of generalized exponents]{Combinatorics of generalized exponents}
\date{July, 2017}
\author{C\'edric Lecouvey and Cristian Lenart}

\begin{abstract}
We give a purely combinatorial proof of the positivity of the stabilized forms
of the generalized exponents associated to each classical root system.\ In
finite type $A_{n-1}$, we rederive the description of the generalized
exponents in terms of crystal graphs without using the combinatorics of
semistandard tableaux or the charge statistic. In finite type $C_{n}$, we
obtain a combinatorial description of the generalized exponents based on the so-called 
distinguished vertices in crystals of type $A_{2n-1}$, which we also connect
to symplectic King tableaux. This gives a combinatorial proof of the positivity
of Lusztig $t$-analogues associated to zero weight spaces in the irreducible
representations of symplectic Lie algebras. We also present three applications of our combinatorial formula, and discuss some implications to relating two type $C$ branching rules. Our methods are expected to extend
to the orthogonal types.
\end{abstract}

\maketitle

\section{Introduction}

Let $\mathfrak{g}$ be a simple Lie algebra over $\mathbb{C}$ of rank $n$ and
$G$ its corresponding Lie group.\ The group $G$ acts on the symmetric algebra
$S(\mathfrak{g})$ of $\mathfrak{g}$, and it was proved by Kostant\ \cite{kost}
that $S(\mathfrak{g})$ factors as $S(\mathfrak{g})=H(\mathfrak{g}%
)\otimes S(\mathfrak{g})^{G}$, where $H(\mathfrak{g})$ is the harmonic part of
 $S(\mathfrak{g)}$. The generalized exponents of
$\mathfrak{g}$, as defined by Kostant \cite{kost}, are the polynomials appearing as the coefficients in the
expansion of the graded character of $H(\mathfrak{g)}$ in the basis of the
Weyl characters.\ It was shown by Hesselink \cite{hes} that these polynomials
coincide, in fact, with the Lusztig $t$-analogues of zero weight multiplicities
in the irreducible finite-dimensional representations of $\mathfrak{g}$. In
particular, they have non-negative integer coefficients, because they are affine
Kazhdan-Lusztig polynomials (see \cite{NR}). Note that the zero weight Lusztig $t$-analogues are the most complex ones. 

For $\mathfrak{g}=\mathfrak{sl}_{n}$, the generalized exponents admit a nice
combinatorial description in terms of the Lascoux-Sch\"{u}tzenberger charge
statistic on semistandard tableaux of zero weight \cite{LSc1}. This
statistic is defined via the cyclage operation on tableaux, which is based on the Schensted
insertion scheme. This combinatorial description extends, in fact, to any
Lusztig $t$-analogue of type $A_{n-1}$, that is possibly associated to a
nonzero weight (also called Kostka polynomials).\ So we have a purely
combinatorial proof of the positivity of their coefficients. It was also
established in \cite{NY} that the Lusztig $t$-analogues in type $A_{n-1}$ are
one-dimensional sums, i.e., some graded multiplicities related to finite-dimensional
representations of quantum groups of affine type $A_{n-1}^{(1)}$. Another
interpretation of the charge statistic in terms of crystals of type $A_{n-1}$
was given later by Lascoux, Leclerc and Thibon in \cite{LLT}.

Despite many efforts during the last three decades, no general combinatorial
proof of the positivity of the Lusztig $t$-analogues is known beyond type $A$.
Nevertheless such proofs have been obtained in some particular cases.\ Notably,
a combinatorial description of the generalized exponents associated to small
representations was given in \cite{Ion} and \cite{Ion2} for any root system.
In \cite{LecShi}, it was established that some Lusztig $t$-analogues of
classical types equal one-dimensional sums for affine quantum groups, which 
generalizes the result of \cite{NY}. Nevertheless, the two families of polynomials
do not coincide beyond type $A$. In \cite{lec} and \cite{Lec2}, charge
statistics based on cyclage on Kashiwara-Nakashima tableaux were defined for classical
types, yielding the desired positivity for particular Lusztig $t$%
-analogues. It is worth mentioning that, in type $C_{n}$, a version of the mentioned 
statistic \cite{lec} permits conjecturally to describe all the Lusztig $t$-analogues in this case.

In \cite{bry}, Brylinsky obtained an algebraic proof of the positivity of any
Lusztig $t$-analogue based on the filtration by a central idempotent of
$\mathfrak{g}$.\ For classical types, this filtration stabilizes
\cite{Hanlon,Lec4}, which yields stabilized versions of these
polynomials. They are formal series in the variable $t$ which, in many respects, 
are more tractable as their finite rank counterparts.

\bigskip

The goal of this paper is twofold. First we give a combinatorial description
of the stabilized version of the generalized exponents and a proof of their
positivity by using the combinatorics of type $A_{+\infty}$ crystal graphs.
This can be regarded as a generalization of results in \cite{LLT} for the
weight zero, and in fact we were able to rederive the latter without any reference to
the charge statistic or the combinatorics of semistandard tableaux. Our
description is in terms of the so-called distinguished vertices in crystal of type
$A_{+\infty}$, but we show that these vertices are in natural bijection with
some generalizations of symplectic King tableaux, which makes the link with stable
Lusztig $t$-analogue more natural. Next, we provide a complete combinatorial
proof of the positivity of the generalized exponents in the non-stable $C_{n}$
case.\ Observe there that the non-stable case is much more involved than the
stable one, essentially because we need a combinatorial description of the non-Levi 
branching from $\mathfrak{gl}_{2n}$ to $\mathfrak{sp}_{2n}$, which is
complicated in general.\ Here we use in a crucial way recent duality results
by Kwon \cite{kwoces,kwoldk} giving a crystal interpretation of the
previous branching and a combinatorial model relevant to its study. We also rely on the complex combinatorics of the bijections realizing the symmetries of type $A$ Littlewood-Richardson coefficients: the combinatorial $R$-matrix and the conjugation symmetry map; both have many different realizations in the literature. We
strongly expect to extend our approach to orthogonal types as soon as all the
results of \cite{kwoldk} will be available for the non-Levi orthogonal branchings.

\bigskip

The paper is organized as follows. In Section~2, we recall the definition of
the generalized exponents and show that, for classical types, they satisfy 
important relations in the ring of formal series in $t$ deduced from Cauchy
and Littlewood identities. In Section~3, we briefly rederive the combinatorial
description of the generalized exponents in type $A_{n-1}$ obtained in
\cite{LLT} without using the results of \cite{LSc1} on the charge. Section~4
is devoted to the combinatorial description of the stabilized form of the
generalized exponents in terms of distinguished tableaux, which we define and study
here. Our approach also permits to extend our results to multivariable versions of the generalized exponents as done in \cite{LLT} for type $A$. In Sections 5 and 6, we give the promised combinatorial description of
the generalized exponents in type $C_{n}$ by using distinguished tableaux
adapted to the finite rank $n$. In Section~5 we do this based on the type $C_{n}$ branching rule due to Sundaram \cite{Sun}, whereas in Section~6 we use Kwon's branching rule; the latter leads to a more explicit description, including one in terms of the symplectic King
tableaux \cite{King}. In Section~7, we give three applications of the description in Section~6: (1) analyzing the growth of the generalized exponents of type $C_n$ with respect to the rank $n$; (2) proving a conjecture related to the construction of the type $C_{n}$ charge in \cite{lec}; (3) determining the smallest power of $t$ in a generalized exponent (note that the largest one is well-known).  The third result turns out to be quite subtle, and it illustrates the combinatorial complexity
of these polynomials. Finally, in Section~8 we raise a question about the possible relationship between the Sundaram and Kwon branching rules.

\bigskip

\noindent\textbf{Acknowledgments:} Both authors thank the RiP program of the
Institut Henri Poincar\'{e} for its invitation to Paris in July 2017, when
this work was completed. The second author was also partially supported by the NSF grant DMS--1362627. We are grateful to Olga Azenhas, for her detailed comments on which Section~\ref{lrcs} is based, to Jae-Hoon Kwon and Sheila Sundaram for valuable input, as well as to Bogdan Ion for asking the question which led to our third application in Section~7. 

\section{Generalized exponents}

\subsection{Background}

Let $\mathfrak{g}_{n}$ be a simple Lie algebra over $\mathbb{C}$ of rank $n$
with triangular decomposition
\[
\mathfrak{g}_{n}\mathfrak{=}\bigoplus\limits_{\alpha\in R_{+}}\mathfrak{g}%
_{\alpha}\oplus\mathfrak{h}\oplus\bigoplus\limits_{\alpha\in R_{+}%
}\mathfrak{g}_{-\alpha}\,,%
\]
so that $\mathfrak{h}$ is the Cartan subalgebra of $\mathfrak{g}_{n}$ and
$R_{+}$ its set of positive roots. The root system $R=R_{+}\sqcup(-R_{+})$ of
$\mathfrak{g}_{n}$ is realized in a real Euclidean space $E$ with inner
product $(\cdot,\cdot)$.\ For any $\alpha\in R,$ we write $\alpha^{\vee}%
=\frac{2\alpha}{(\alpha,\alpha)}$ for its coroot. Let $S\subset R_{+}$ be the
subset of simple roots and $Q_{+}$ the cone generated by $S$. The set $P$ of
integral weights for $\mathfrak{g}_{n}$ satisfies $(\beta,\alpha^{\vee}%
)\in\mathbb{Z}$ for any $\beta\in P$ and $\alpha\in R$. We write
$P_{+}=\{\beta\in P\mid(\beta,\alpha^{\vee})\geq0$ for any $\alpha\in S\}$ for
the cone of dominant weights of $\mathfrak{g}_{n}$, and denote by $\omega
_{1},\ldots,\omega_{n}$ its fundamental weights. Let $W$ be the Weyl group of
$\mathfrak{g}_{n}$ generated by the reflections $s_{\alpha}$ with $\alpha\in
S$, and write $\ell$ for the corresponding length function.

The graded character of the symmetric algebra $S(\mathfrak{g}_{n})$ of
$\mathfrak{g}_{n}$ is defined by%
\[
\mathrm{char}_{t}(S(\mathfrak{g}_{n}))=\prod_{\delta\text{ weight of
}\mathfrak{g}_{n}}\frac{1}{1-te^{\delta}}=\frac{1}{(1-t)^{n}}\prod_{\alpha\in
R}\frac{1}{1-te^{\alpha}}\,.
\]
By a classical theorem due to Kostant, the graded character of the harmonic
part of the symmetric algebra $S(\mathfrak{g}_{n})$ satisfies%
\[
\mathrm{char}_{t}(H(\mathfrak{g}_{n}))=\frac{\prod_{i=1}^{n}(1-t^{d_{i}}%
)}{(1-t)^{n}}\prod_{\alpha\in R}\frac{1}{1-te^{\alpha}}=\prod_{i=1}%
^{n}(1-t^{d_{i}})\mathrm{char}_{t}(S(\mathfrak{g}_{n}))\,,
\]
where we have $d_{i}=m_{i}+1$, for $i=1,\ldots,n$, and $m_{1},\ldots,m_{n}$
are the (classical) exponents of $\mathfrak{g}_{n}$. On the other hand, it is known (see
\cite{hes}) that $\mathrm{char}_{t}(H(\mathfrak{g}_{n}))$ coincides with the
Hall-Littlewood polynomial $Q_{0}^{\prime}$, namely we have
\begin{equation*}
\mathrm{char}_{t}(H(\mathfrak{g}_{n}))=Q_{0}^{\prime}=W_{0}(t)\prod_{\alpha\in
R}\frac{1}{1-te^{\alpha}}=\sum_{\lambda\in P_{+}}K_{\lambda,0}^{\mathfrak{g}%
_{n}}(t)s_{\lambda}^{\mathfrak{g}_{n}}\,,%
\end{equation*}
where
\[
W_{0}(t)=\sum_{w\in W}t^{\ell(w)}\,,%
\]
and $s_{\lambda}^{\mathfrak{g}_{n}}$ is the Weyl character associated to the
finite-dimensional irreducible representation $V(\lambda)$ of $\mathfrak{g}%
_{n}$ with highest weight $\lambda$. In particular, we have the identity%
\[
W_{0}(t)=\prod_{i=1}^{n}\frac{1-t^{d_{i}}}{1-t}\,.
\]

The polynomials $K_{\lambda,0}^{\mathfrak{g}_{n}}(t)$ are the generalized 
exponents of $\mathfrak{g}_{n}$, and they coincide with the Lusztig
$t$-analogues associated to the zero weight subspaces in the representations
$V(\lambda)$. We thus have
\[
K_{\lambda,0}^{\mathfrak{g}_{n}}(t)=\sum_{w\in W}(-1)^{\ell(w)}\emph{P}%
_{t}(w(\lambda+\rho)-\rho)\,,
\]
where $\rho$ is half the sum of the positive roots, and $\emph{P}_{t}$ is the
$t$-Kostant partition function defined by
\[
\prod_{\alpha\in R_{+}}\frac{1}{1-te^{\alpha}}=\sum_{\beta\in Q_{+}}%
\emph{P}_{t}(\beta)e^{\beta}\,.
\]
The classical exponents $m_1,\ldots,m_n$ correspond to the adjoint representation of $\mathfrak{g}_{n}$, namely we have
\[K_{\widetilde{\alpha},0}^{\mathfrak{g}_{n}}(t)=\sum_{i=1}^n t^{m_i}\,,\]
where $\widetilde{\alpha}$ is the highest root in $R_+$. 

\subsection{Classical types}

Recall the values of the classical exponents in types $A-D$: %
\[%
\begin{array}
[c]{cc}%
\text{type }X & \text{exponents}\\
A_{n-1} & 1,2,\ldots,n-1\\
B_{n} & 1,3,\ldots,2n-1\\
C_{n} & 1,3,\ldots,2n-1\\
D_{n} & 1,3,\ldots,2n-3\text{ and }n-1.
\end{array}
\]

In classical types, $\mathrm{char}_{t}(S(\mathfrak{g}_{n}))$ is easy to
compute. Let $\mathcal{P}_{n}$ be the set of partitions with at most $n$ parts, 
and $\mathcal{P}$ the set of all partitions. The rank of the partition $\gamma$ is defined as the sum of its parts, and is denoted by 
$\left\vert \gamma\right\vert$.

In type $A_{n-1}$, we start from the Cauchy identity%
\[
\prod_{1\leq i,j\leq n}\frac{1}{1-tx_{i}y_{j}}=\sum_{\gamma\in\mathcal{P}_{n}%
}t^{\left\vert \gamma\right\vert }s_{\gamma}(x)s_{\gamma}(y)\footnote{Here
$s_{\nu}(x)$ stands for the ordinary Schur function in the variables
$x_{1},\ldots,x_{n}$.}.
\]
By setting $y_{i}=\frac{1}{x_{i}}$ for any $i=1,\ldots,n$, and by considering the
images of the symmetric polynomials in $R^{A_{n-1}}=\mathrm{Sym}[x_{1}%
,\ldots,x_{n}]/(x_{1}\cdots x_{n}-1)$, we get
\begin{align}
\mathrm{char}_{t}(S(\mathfrak{sl}_{n}))&=(1-t)\sum_{\gamma\in\mathcal{P}_{n}%
}t^{\left\vert \gamma\right\vert }s_{\gamma}(x)s_{\gamma}(x^{-1}%
)=(1-t)\sum_{\gamma\in\mathcal{P}_{n}}t^{\left\vert \gamma\right\vert
}s_{\gamma}s_{\gamma^{\ast}}=\label{charSA}\\
&=(1-t)\sum_{\gamma\in\mathcal{P}_{n}}t^{\left\vert \gamma\right\vert }%
\sum_{\lambda\in\mathcal{P}_{n-1}}c_{\gamma,\gamma^{\ast}}^{\lambda}%
s_{\lambda}(x)\,.\nonumber
\end{align}
Here $\gamma^{\ast}=-w_{\circ}(\gamma)$, where $w_{\circ}$ is the permutation of
maximal length in $S_{n}$, and we use the same notation for a symmetric
polynomial and its image in $R^{A_{n-1}}$.\ Recall also that the partitions of
$\mathcal{P}_{n-1}$ are in one-to-one correspondence with the dominant weights
of $\mathfrak{sl}_{n}.$ More precisely, we associate to the dominant weight
$a_{1}\omega_{1}+\cdots+a_{n-1}\omega_{n-1}$ the partition $\lambda=(1^{a_{1}},\ldots,(n-1)^{a_{n-1}})'$, where $\mu'$ denotes the conjugate of $\mu$. It is also worth mentioning here that two partitions
in $\mathcal{P}_{n}$ whose conjugates have the same parts less than $n$ correspond to the
same dominant weight of $\mathfrak{sl}_{n}$. So the coefficients
$c_{\gamma,\gamma^{\ast}}^{\lambda}$ are not properly Littlewood-Richardson
coefficients, but only tensor multiplicities corresponding to the decomposition
of $V(\gamma)\otimes V(\gamma^{\ast})$ into irreducible components. Similarly
$s_{\lambda}(x)$ is not properly a Schur polynomial but belongs to $R^{A_{n-1}}$
(see Section~\ref{Sec_chcrys}).

\bigskip

For any positive integer $m$, define $\mathcal{P}_{m}^{(2)}$ as the set of
partitions of the form $2\kappa$ with $\kappa\in\mathcal{P}_{m}$, and
$\mathcal{P}_{m}^{(1,1)}$ as the subset of $\mathcal{P}_{m}$ containing the
partitions of the form $(2\kappa)^{\prime}$ with $\kappa\in\mathcal{P}$. 

In type $B_{n}$, we start from the Littlewood identity \cite{Li}
\[
\prod_{1\leq i<j\leq2n+1}\frac{1}{1-ty_{i}y_{j}}=\sum_{\nu\in\mathcal{P}%
_{2n+1}^{(1,1)}}t^{\left\vert \nu\right\vert /2}s_{\nu}(y)\,,
\]
and we specialize $y_{2n+1}=1,y_{2i-1}=x_{i}$, and $y_{2i}=\frac{1}{x_{i}}$, for
any $i=1,\ldots,n$.\ This gives
\[
\mathrm{char}_{t}(S(\mathfrak{so}_{2n+1}))=\sum_{\nu\in\mathcal{P}%
_{2n+1}^{(1,1)}}t^{\left\vert \nu\right\vert /2}\sum_{\lambda\in
\mathcal{P}_{n}}c_{\nu}^{\lambda}(\mathfrak{so}_{2n+1})s_{\lambda
}^{\mathfrak{so}_{2n+1}}\,,%
\]
where $c_{\nu}^{\lambda}(\mathfrak{so}_{2n+1})$ is the branching coefficient
corresponding to the restriction from $\mathfrak{gl}_{2n+1}$ to $\mathfrak{so}%
_{2n+1}$. Similarly, we can consider the identities%
\[
\prod_{1\leq i<j\leq2n}\frac{1}{1-ty_{i}y_{j}}=\sum_{\nu\in\mathcal{P}%
_{2n}^{(1,1)}}t^{\left\vert \nu\right\vert /2}s_{\nu}(y)\;\;\;\text{ and }\;\;\;%
\prod_{1\leq i\leq j\leq2n}\frac{1}{1-ty_{i}y_{j}}=\sum_{\nu\in\mathcal{P}%
_{2n}^{(2)}}t^{\left\vert \nu\right\vert /2}s_{\nu}(y)\,.
\]
They permit to write%
\begin{align*}
\mathrm{char}_{t}(S(\mathfrak{sp}_{2n})) &  =\sum_{\nu\in\mathcal{P}%
_{2n}^{(2)}}t^{\left\vert \nu\right\vert /2}\sum_{\lambda\in\mathcal{P}_{n}%
}c_{\nu}^{\lambda}(\mathfrak{sp}_{2n})s_{\lambda}^{\mathfrak{sp}_{2n}}\;\;\;\text{ and}\\
\mathrm{char}_{t}(S(\mathfrak{so}_{2n})) &  =\sum_{\nu\in\mathcal{P}%
_{2n}^{(1,1)}}t^{\left\vert \nu\right\vert /2}\sum_{\lambda\in\mathcal{P}_{n}%
}c_{\nu}^{\lambda}(O_{2n})s_{\lambda}^{O_{2n}}\,.
\end{align*}
Here we should, in fact, consider the character $s_{\lambda}^{O_{2n}}$ of
the $O(2n)$-module $V^{O(2n)}(\lambda)$ parametrized by the partition
$\lambda$.\ When $\lambda_{n}=0$, we have $s_{\lambda}^{O_{2n}}=s_{\lambda
}^{\mathfrak{so}_{2n}}$. Nevertheless, when $\lambda_{n}>0$, $V^{O(2n)}%
(\lambda)$ decomposes as the sum of two irreducible $SO(2n)$-modules whose
highest weights correspond via the Dynkin diagram involution $\iota$ flipping
the nodes $n-1$ and $n.$ For $1\leq i\leq n-1$, define $a_{i}$ as the number
of columns of height $i$ in $\lambda$, and $a_{n}=2\lambda_{n}+a_{n-1}$. We
then have
\begin{equation}
s_{\lambda}^{O_{2n}}=s_{\omega(\lambda)}^{\mathfrak{so}_{2n}}+s_{\iota
(\omega(\lambda))}^{\mathfrak{so}_{2n}}\label{CharO(2n))}\,,%
\end{equation}
where $\omega(\lambda)=\sum_{i=1}^{n}a_{i}\omega_{i}$.

Since we have%
\[
\mathrm{char}_{t}(H(\mathfrak{g}_{n}))=\frac{\prod_{i=1}^{n}(1-t^{d_{i}}%
)}{(1-t)^{n}}\prod_{\alpha\in R}\frac{1}{1-te^{\alpha}}=\prod_{i=1}%
^{n}(1-t^{d_{i}})\mathrm{char}_{t}(S(\mathfrak{g}_{n}))\,,
\]
we can write
\begin{equation}
\frac{1}{\prod_{i=1}^{n}(1-t^{d_{i}})}\mathrm{char}_{t}(H(\mathfrak{g}%
_{n}))=\sum_{\lambda\in P_{+}}\frac{K_{\lambda,0}(t)}{\prod_{i=1}%
^{n}(1-t^{d_{i}})}s_{\lambda}^{\mathfrak{g}_{n}}=\mathrm{char}_{t}%
(S(\mathfrak{g}_{n}))\,\text{.}\label{(H)}%
\end{equation}
So we get the following simple expressions for the formal series
$\frac{K_{\lambda,0}(t)}{\prod_{i=1}^{n}(1-t^{d_{i}})}$.

\begin{proposition}\label{br-rules}
We have the following identities.

\begin{enumerate}
\item In type $A_{n-1}$, for any $\lambda\in\mathcal{P}_{n-1}$, we have%
\begin{equation}
\frac{K_{\lambda,0}^{\mathfrak{sl}_{n}}(t)}{\prod_{i=1}^{n}(1-t^{i})}%
=\sum_{\gamma\in\mathcal{P}_{n}}t^{\left\vert \gamma\right\vert }%
c_{\gamma,\gamma^{\ast}}^{\lambda}\,.\footnote{The factor $(1-t)$ in
(\ref{charSA}) gives the missing \textquotedblleft$d_{i}=1$\textquotedblright%
\ in type $A_{n-1}.$}\label{relA}%
\end{equation}

\item In type $B_{n},$ for any $\lambda\in\mathcal{P}_{n}$, we have
\[
\frac{K_{\lambda,0}^{\mathfrak{so}_{2n+1}}(t)}{\prod_{i=1}^{n}(1-t^{2i})}%
=\sum_{\nu\in\mathcal{P}_{2n+1}^{(1,1)}}t^{\left\vert \nu\right\vert /2}%
c_{\nu}^{\lambda}(\mathfrak{so}_{2n+1})\,.\footnote{Here that the partition
$\lambda$ can have an odd rank.}%
\]

\item In type $C_{n}$, for any $\lambda\in\mathcal{P}_{n}$, we have
\[
\frac{K_{\lambda,0}^{\mathfrak{sp}_{2n}}(t)}{\prod_{i=1}^{n}(1-t^{2i})}%
=\sum_{\nu\in\mathcal{P}_{2n}^{(2)}}t^{\left\vert \nu\right\vert /2}c_{\nu
}^{\lambda}(\mathfrak{sp}_{2n})\,.
\]

\item In type $D_{n}$, for any $\lambda\in\mathcal{P}_{n}$, we have
\[
\frac{K_{\lambda,0}^{O(2n)}(t)}{(1-t^{n})\prod_{i=1}^{n-1}(1-t^{2i})}%
=\sum_{\nu\in\mathcal{P}_{2n}^{(1,1)}}t^{\left\vert \nu\right\vert /2}c_{\nu
}^{\lambda}(O_{2n})\,.
\]

\end{enumerate}
\end{proposition}

For type $D_{n}$, the dominant weights appearing in (\ref{(H)}) are not
necessarily partitions, whereas this is the case in Assertion 4 of the previous
proposition. So here we have in fact to write
\begin{align*}
\sum_{\omega\in P_{+}}\frac{K_{\omega,0}^{\mathfrak{so}_{2n}}(t)}{\prod
_{i=1}^{n}(1-t^{d_{i}})}s_{\omega}^{\mathfrak{so}_{2n}}&=\sum_{\lambda\in
P_{+},\lambda\in\mathcal{P}_{n-1}}\frac{K_{\lambda,0}^{\mathfrak{so}_{2n}}%
(t)}{\prod_{i=1}^{n}(1-t^{d_{i}})}s_{\lambda}^{O_{2n}}+\\
&+\sum_{\omega\in
P_{+},\omega\notin\mathcal{P}_{n-1}}\frac{K_{\omega,0}^{\mathfrak{so}_{2n}%
}(t)s_{\omega}^{\mathfrak{so}_{2n}}+K_{\iota(\omega),0}^{\mathfrak{so}_{2n}%
}(t)s_{\iota(\omega)}^{\mathfrak{so}_{2n}}}{\prod_{i=1}^{n}(1-t^{d_{i}})}\\
&=\sum_{\lambda\in\mathcal{P}_{n}}\frac{K_{\lambda,0}^{O(2n)}(t)}{\prod
_{i=1}^{n}(1-t^{d_{i}})}s_{\lambda}^{O_{2n}}\,,%
\end{align*}
where
\[
K_{\lambda,0}^{O(2n)}(t)=K_{\omega(\lambda),0}^{\mathfrak{so}_{2n}%
}(t)=K_{\iota(\omega(\lambda)),0}^{\mathfrak{so}_{2n}}(t)
\]
for any partition $\lambda\in\mathcal{P}_{n}\setminus\mathcal{P}_{n-1}$ and
$\omega(\lambda)$ defined as in (\ref{CharO(2n))}).

\bigskip

The notation $K_{\lambda,0}^{\mathfrak{sl}_{n}}(t)$ is a little unusual in
type $A_{n-1}$, where the polynomials $K_{\lambda,0}^{\mathfrak{sl}_{n}}(t)$
coincide with the Kostka polynomials, which are usually labeled by pairs of
partitions with the same rank (that is, by using the weights of $\mathfrak{gl}%
_{n}$ rather than those of $\mathfrak{sl}_{n}$). When $K_{\lambda
,0}^{\mathfrak{sl}_{n}}(t)\neq0$, the rank of $\lambda$ should in particular
be a multiple of $n$. Also the sum in the right-hand side of Assertion 1 is in
fact infinite. Indeed, to the weight $\lambda$ correspond an infinite
number of partitions, since adding columns of height $n$ to a Young diagram
does not modify the corresponding weight of $\mathfrak{sl}_{n}$.

We have then by a theorem of Lascoux and Sch\"{u}tzenberger \cite{LSc1}
\[
K_{\lambda,0}^{\mathfrak{sl}_{n}}(t)=\sum_{T\in SST(\lambda)_{0}%
}t^{\mathrm{ch}_{n}(T)}\,,%
\]
where $SST(\lambda)_{0}$ is the set of semistandard tableaux labeled by
letters of $\{1<\cdots<n\}$ of weight $\mu=(a,\ldots,a)=0$ (i.e. each letter
$i$ appear $a$ times in $T$) where $a=\left\vert \lambda\right\vert/n$, and $\mathrm{ch}_{n}(T)$ is the charge statistic
evaluated on $T$.\ Recall that this charge statistic is defined by rather
involved combinatorial operation such as cyclage on tableaux.

\subsection{Stable versions}

\label{Section_Stable}When the ranks of the classical root systems considered
go to infinity, the previous relations simplify. In particular, for $n$ sufficiently large, we have
\[
c_{\nu}^{\lambda}(\mathfrak{so}_{2n+1})=\sum_{\delta\in\mathcal{P}}%
c_{\lambda,2\delta}^{\nu}\,,\;\;\;c_{\nu}^{\lambda}(\mathfrak{sp}_{2n}%
)=\sum_{\delta\in\mathcal{P}}c_{\lambda,(2\delta)^{\prime}}^{\nu}\,,\;\;\text{ and }\;\;c_{\nu}^{\lambda}(\mathfrak{so}_{2n})=\sum_{\delta\in\mathcal{P}}%
c_{\lambda,2\delta}^{\nu}\,.
\]
Observe that, for $\mathfrak{g}=\mathfrak{so}_{2n+1}$, this implies in
particular that $c_{\nu}^{\lambda}(\mathfrak{so}_{2n+1})=0$ when the
ranks of $\lambda$ and $\nu$ do not have the same parity, which is false in
general. Thus we get the relations%
\begin{align*}
\frac{K_{\lambda,0}^{B_\infty}(t)}{\prod_{i=1}^{\infty}(1-t^{2i})}&=\sum_{\nu
\in\mathcal{P}^{(1,1)}}\sum_{\delta\in\mathcal{P}^{(2)}}t^{\left\vert
\nu\right\vert /2}c_{\lambda,\delta}^{\nu}\;\;\;\text{ in type }B_{\infty}\text{
when }\left\vert \lambda\right\vert \text{ is even}\,,\\
\frac{K_{\lambda,0}^{C_\infty}(t)}{\prod_{i=1}^{\infty}(1-t^{2i})}&=\sum_{\nu
\in\mathcal{P}^{(2)}}\sum_{\delta\in\mathcal{P}^{(1,1)}}t^{\left\vert
\nu\right\vert /2}c_{\lambda,\delta}^{\nu}\;\;\;\text{ in type }C_{\infty}\,,\\
\frac{K_{\lambda,0}^{D_\infty}(t)}{\prod_{i=1}^{\infty}(1-t^{2i})}&=\sum_{\nu
\in\mathcal{P}^{(1,1)}}\sum_{\delta\in\mathcal{P}^{(2)}}t^{\left\vert
\nu\right\vert /2}c_{\lambda,\delta}^{\nu}\;\;\;\text{ in type }D_{\infty}\,.
\end{align*}
In particular, this gives
\begin{equation}
K_{\lambda,0}^{B_\infty}(t)=K_{\lambda,0}^{D_\infty}(t)\;\;\text{ and }\;\;K_{\lambda,0}%
^{B_\infty}(t)=K_{\lambda^{\prime},0}^{C_\infty}(t)\,.\label{StableRedC}%
\end{equation}
All these stabilized forms are in fact formal power series in $t$ equal to zero when
the rank of $\lambda$ is odd (see \cite{Lec4}). The previous identities permit
to restrict to the study of the stabilized formal series $K_{\lambda,0}^{C_\infty}(t)$ when
$\lambda$ runs over the set of partitions with even rank. We are going to see
that stabilized form of the generalized exponents are easier to handle than their finite rank
counterparts.\ Observe also that stabilized versions of Lusztig $t$-analogues
\cite{Lec4} exist in general (that is, for non-zero weights) in connection with 
the stabilization of the Brylinski filtration. Finally, in type $A$, the Kostka
polynomials $K_{\lambda,0}^{\mathfrak{sl}_{n}}(t)$ stabilize to zero when $n$
becomes greater than the rank of $\lambda$.

\section{Charge in type $A_{n-1}$ and crystal graphs}

\label{Sec_chcrys}We are now going to explain how the interpretation of the
charge for zero weight tableaux in terms of crystals obtained in \cite{LLT}
naturally emerges from (\ref{relA}), without any reference to cyclage. In
particular, we obtain a direct proof of the positivity of the polynomials
$K_{\lambda,0}^{A_{n-1}}(t)$; for simplicity, we drop the superscript $A_{n-1}$. 
We also refer to \cite{kash} for complements on Kashiwara crystal basis theory.

\noindent\textbf{Step 1 : }Observe that $c_{\gamma,\gamma^{\ast}}^{\lambda
}=c_{\kappa,\kappa^{\ast}}^{\lambda}$ for $\gamma,\kappa$ in $\mathcal{P}_{n}$
whose conjugates differ only by their parts equal to $n$.\ So by decomposing each $\kappa
\in\mathcal{P}_{n}$ as $\kappa=(\gamma,n^{m})$, we get
\[
\sum_{\kappa\in\mathcal{P}_{n}}t^{\left\vert \kappa\right\vert }%
c_{\kappa,\kappa^{\ast}}^{\lambda}=\sum_{\gamma\in\mathcal{P}_{n-1}%
}t^{\left\vert \kappa\right\vert }c_{\gamma,\gamma^{\ast}}^{\lambda}\sum
_{m=0}^{+\infty}(t^{n})^{m}=\frac{1}{1-t^{n}}\sum_{\gamma\in\mathcal{P}_{n-1}%
}t^{\left\vert \kappa\right\vert }c_{\gamma,\gamma^{\ast}}^{\lambda}\,.
\]
Therefore, (\ref{relA}) can be rewritten in the form
\begin{equation}
\frac{K_{\lambda,0}(t)}{\prod_{i=1}^{n-1}(1-t^{i})}=\sum_{\gamma\in
\mathcal{P}_{n-1}}t^{\left\vert \gamma\right\vert }c_{\gamma,\gamma^{\ast}%
}^{\lambda} \label{relAs}\,,%
\end{equation}
where now all the partitions are in one-to-one correspondence with weights of
$\mathfrak{sl}_{n}$.

\noindent\textbf{Step 2 : }Recall that $R^{A_{n-1}}$ is endowed with the
scalar product $\langle\cdot,\cdot\rangle$ defined by
\[
\langle f,g\rangle=[fa_{\rho}\overline{ga_{\rho}}]_{0}\,,%
\]
and we then have $\langle s_{\lambda},s_{\mu}\rangle=\delta_{\lambda,\mu}$. It
follows that the adjoint of the multiplication by $s_{\lambda}$ in
$R^{A_{n-1}}$ for this scalar product is the multiplication by $s_{\lambda
^{\ast}}$. This gives
\[
c_{\gamma,\gamma^{\ast}}^{\lambda}=\langle s_{\gamma}s_{\gamma^{\ast}%
},s_{\lambda}\rangle=\langle s_{\gamma},s_{\gamma}s_{\lambda}\rangle
=c_{\gamma,\lambda}^{\gamma}\,.
\]
\noindent\textbf{Step 3 : }For any $\lambda\in\mathcal{P}_{n-1},$ write
$B(\lambda)$ for the crystal graph of the irreducible $\mathfrak{sl}_{n}%
$-module of highest weight $\lambda$.\ Let $b_{\lambda}$ be the highest weight
vertex of $B(\lambda)$. For any vertex $b\in B(\lambda)$, set
\[
\boldsymbol{\varepsilon}(b)=\sum_{i=1}^{n-1}\varepsilon_{i}(b)\omega_{i}%
\in\mathcal{P}_{n-1}\text{.}%
\]
Also given $\kappa,\delta$ in $\mathcal{P}_{n-1}$, write $\kappa\leq\delta$
when $\delta-\kappa$ is a dominant weight. We know that
\[
c_{\gamma,\lambda}^{\gamma}=\mathrm{card}\left\{  b_{\gamma}\otimes b\in
B(\gamma)\otimes B(\lambda)\mid\mathrm{wt}(b)=0\;\;\text{ and }\;\;%
\boldsymbol{\varepsilon}(b)\leq\gamma\right\}  .
\]
So we have in fact $c_{\gamma,\lambda}^{\gamma}=\mathrm{card}(B_{\lambda
}(\gamma))$, where
\[
B_{\lambda}(\gamma)=\{b\in B(\lambda)\mid\mathrm{wt}(b)=0\;\;\text{ and }\;\;\boldsymbol{\varepsilon}(b)\leq\gamma\}\,.
\]
Now for any $b\in B(\lambda)_{0}$, that is, $b\in B(\lambda)$ such that
$\mathrm{wt}(b)=0$, set%
\[
S(b)=\{\gamma\in\mathcal{P}_{n-1}\mid\boldsymbol{\varepsilon}(b)\leq\gamma\}\,.
\]
We have in fact
\[
S(b)=\boldsymbol{\varepsilon}(b)+\mathcal{P}_{n-1}\,,%
\]
that is, $\gamma\in S(b)$ if and only if there exists $\kappa\in\mathcal{P}%
_{n-1}$ such that $\gamma=\boldsymbol{\varepsilon}(b)+\kappa$.

\noindent\textbf{Step 4 : }Write
\begin{align*}
\sum_{\gamma\in\mathcal{P}_{n-1}}t^{\left\vert \gamma\right\vert }%
c_{\gamma,\gamma^{\ast}}^{\lambda}&=\sum_{b\in B(\lambda)_{0}}\sum_{\gamma\in
S(b)}t^{\left\vert \gamma\right\vert }=\sum_{b\in B(\lambda)_{0}}\sum
_{\kappa\in\mathcal{P}_{n-1}}t^{\left\vert \kappa\right\vert +\left\vert
\boldsymbol{\varepsilon}(b)\right\vert }=\\
&=\sum_{b\in B(\lambda)_{0}}t^{\left\vert \boldsymbol{\varepsilon}(b\right\vert
}\sum_{\kappa\in\mathcal{P}_{n-1}}t^{\left\vert \kappa\right\vert }=\frac
{1}{\prod_{i=1}^{n-1}(1-t^{i})}\sum_{b\in B(\lambda)_{0}}t^{\left\vert
\boldsymbol{\varepsilon}(b\right\vert }\,.
\end{align*}
In view of (\ref{relAs}), this gives the following result.

\begin{theorem}\label{charge-a}
For any partition $\lambda\in\mathcal{P}_{n-1}$, we have
\[
K_{\lambda,0}(t)=\sum_{b\in B(\lambda)_{0}}t^{\left\vert
\boldsymbol{\varepsilon}(b)\right\vert }\,.
\]
\end{theorem}

\begin{remark}{\rm 
Since $b\in B(\lambda)_{0}$, we have for any $i=1,\ldots,n-1$ that
$\varepsilon_{i}(b)=\varphi_{i}(b).$ Moreover
\[
\left\vert \boldsymbol{\varepsilon}(b)\right\vert =\sum_{i=1}^{n-1}%
i\varepsilon_{i}(b)\,,
\]
so the previous expression of $K_{\lambda,0}(t)$ is the same as that obtained
in \cite{LLT}. The interesting point is that it emerges directly from our
computations, and does not use the definition of the charge (as in \cite{LLT})
given by Lascoux and Sch\"{u}tzenberger in terms of cyclage of tableaux or
indices on letters of words.}
\end{remark}

\begin{remark}{\rm 
This also permits to recover the multivariable version, defined by
\[
\frac{K_{\lambda,0}(t_{1},\ldots,t_{n-1})}{\prod_{i=1}^{n-1}(1-t_{i})}%
=\sum_{\gamma\in\mathcal{P}_{n-1}}\prod_{i=1}^{n-1}t_{i}^{a_{i}(\gamma
)}c_{\gamma,\gamma^{\ast}}^{\lambda}\,,%
\]
where $\gamma=\sum_{i=1}^{n-1}a_{i}(\gamma)\omega_{i}$. Namely, we have
\[
K_{\lambda,0}(t_{1},\ldots,t_{n-1})=\sum_{b\in B(\lambda)_{0}}\prod
_{i=1}^{n-1}t_{i}^{\varepsilon_{i}(b)}\,.%
\]
}
\end{remark}

\section{Stabilized generalized exponents and crystal graphs of type
$A_{+\infty}$}

\label{Sec_Cinfinite}We are going to explain the way in which the formula%
\[
\frac{K_{\lambda,0}^{C_\infty}(t)}{\prod_{i=1}^{\infty}(1-t^{2i})}=\sum_{\nu
\in\mathcal{P}^{(2)}}\sum_{\delta\in\mathcal{P}^{(1,1)}}t^{\left\vert
\nu\right\vert /2}c_{\lambda,\delta}^{\nu}\;\;\;\text{ in type }C_{\infty}%
\]
can be obtained from the combinatorics of crystals of type $A_{+\infty}$, which leads to a combinatorial proof of the positivity of the stabilized generalized exponents (or stabilized Lusztig $t$-analogues). In
particular, this will provide a combinatorial description of $K_{\lambda,0}^{C_\infty}(t)$,
and thus a similar description of $K_{\lambda,0}^{B_\infty}(t)$ and
$K_{\lambda,0}^{D_\infty}(t)$, by (\ref{StableRedC}). Furthermore, this will give a flavor of
the methods we will employ in the non-stable type $C_{n}$ case.

\subsection{Crystal of type $A_{+\infty}$}

\label{sub_crystA}Recall that crystals of type $A_{+\infty}$ are those
associated to the infinite Dynkin diagram%
\[
\overset{1}{\circ}-\overset{2}{\circ}-\overset{3}{\circ}\cdot\cdot\cdot
\]
The partitions label the dominant weights of $\mathfrak{sl}_{+\infty}$. If we
denote by $(\omega_{i})_{\geq1}$ the sequence of fundamental weights of $\mathfrak{sl}%
_{+\infty}$, we have for any partition $\lambda\in\mathcal{P}$%
\[
\lambda=\sum_{i}a_{i}\omega_{i}\,,%
\]
where $a_{i}$ is the number of columns with height $i$ in the Young diagram of
$\lambda$.

To each partition $\lambda$ corresponds the crystal $B(\lambda)$ of the
irreducible infinite-dimensional representation of $\mathfrak{sl}_{+\infty}$
parametrized by $\lambda$. A classical model for $B(\lambda)$ is that of
semistandard tableaux of shape $\lambda$ on the infinite alphabet
$\mathbb{Z}_{>0}=\{1<2<3<\cdots\}$. Given $b\in B(\lambda)$, we define%
\[
\boldsymbol{\varepsilon}(b)=\sum_{i=1}^{+\infty}\varepsilon_{i}(b)\omega
_{i}\;\;\text{ and }\;\;\boldsymbol{\varphi}(b)=\sum_{i=1}^{+\infty}\varepsilon
_{i}(b)\omega_{i}\,,%
\]
where both sums are in fact finite. The weight of $b\in B(\lambda)$
then verifies $\mathrm{wt}(b)=\boldsymbol{\varphi}(b)-\boldsymbol{\varepsilon
}(b)$.

\subsection{Combinatorial preliminaries}

In the sequel we consider the order $\leq$ on $\mathcal{P}$ such that
$\lambda\leq\mu$ if and only if $\mu-\lambda\in P_{+}^{\infty}$, that is, 
$\mu-\lambda$ decomposes in the basis of the $\omega_{i}$'s with non-negative
integer coefficients.

The partitions in $\mathcal{P}^{(2)}$ (resp. in $\mathcal{P}^{(1,1)}$) are
those which can be tiled with horizontal (resp. vertical)
dominoes.\ Equivalently, a partition $\kappa$ belongs to $\mathcal{P}^{(2)}$
(resp. $\mathcal{P}^{(1,1)}$) if and only if the number of columns (resp.
rows) of fixed height (resp. length) is even. So
\[
\kappa\in\mathcal{P}^{(2)}\;\;\Longleftrightarrow\;\;\kappa=\sum_{i}2a_{i}\omega
_{i}\;\;\text{ and }\;\;\kappa\in\mathcal{P}^{(1,1)}\;\;\Longleftrightarrow\;\;\kappa=\sum
_{i}a_{i}\omega_{2i}\,.
\]
Set $\mathcal{P}^{\boxplus}=\mathcal{P}^{(2)}\cap\mathcal{P}^{(1,1)}$.\ It
follows that
\[
\kappa\in\mathcal{P}^{\boxplus}\;\;\Longleftrightarrow\;\;\kappa=\sum_{i}2a_{i}%
\omega_{2i}\,,%
\]
that is, $\lambda$ decomposes in terms of the fundamental weights $\omega_{2i}$ with
even coefficients. In the general case of a partition $\kappa\in\mathcal{P}$
written as
\[
\kappa=\sum_{i}a_{i}\omega_{i}\,,%
\]
we define%
\[
\kappa_{\boxplus}=\sum_{i}(a_{2i}-(a_{2i}\;\mathrm{mod}\;2))\omega_{2i}\;\;\text{ and }\;\;\kappa^{\boxplus}=\kappa-\kappa_{\boxplus}=\sum_{i}a_{2i+1}\omega_{2i+1}%
+\sum_{i}(a_{2i}\;\mathrm{mod}\;2)\omega_{2i}\,.
\]
So $\kappa_{\boxplus}$ and $\kappa^{\boxplus}$ are partitions and
$\kappa_{\boxplus}\in\mathcal{P}^{\boxplus}$.

\begin{example}{\rm 
Consider
\[
\kappa=%
\begin{tabular}
[c]{|l|l|llll}\hline
&  &  & \multicolumn{1}{|l}{} & \multicolumn{1}{|l}{} & \multicolumn{1}{|l|}{}%
\\\hline
&  &  & \multicolumn{1}{|l}{} & \multicolumn{1}{|l}{} & \multicolumn{1}{|l|}{}%
\\\hline
&  &  & \multicolumn{1}{|l}{} &  & \\\cline{1-3}
&  &  &  &  & \\\cline{1-2}
&  &  &  &  & \\\cline{1-2}
&  &  &  &  & \\\cline{1-2}%
\end{tabular}\,.
\]
Then
\[
\kappa_{_{\boxplus}}=%
\begin{tabular}
[c]{|l|l|ll}\hline
&  &  & \multicolumn{1}{|l|}{}\\\hline
&  &  & \multicolumn{1}{|l|}{}\\\hline
&  &  & \\\cline{1-2}
&  &  & \\\cline{1-2}
&  &  & \\\cline{1-2}
&  &  & \\\cline{1-2}%
\end{tabular}
\;\;\text{ and }\;\;\kappa^{_{\boxplus}}=%
\begin{tabular}
[c]{|l|l}\hline
& \multicolumn{1}{|l|}{}\\\hline
& \multicolumn{1}{|l|}{}\\\hline
& \\\cline{1-1}%
\end{tabular}
\,.
\]
}
\end{example}

We denote by $P_{(2)}^{\infty}$ and $P_{(1,1)}^{\infty}$ the sublattices of
$P=%
{\textstyle\bigoplus\limits_{i\geq1}}
\mathbb{Z\omega}_{i}$ defined by
\[
P_{(2)}^{\infty}=%
{\textstyle\bigoplus\limits_{i\geq1}}
2\mathbb{Z\omega}_{i}\;\;\text{ and }\;\;P_{(1,1)}^{\infty}=%
{\textstyle\bigoplus\limits_{i\geq1}}
\mathbb{Z\omega}_{2i}\,.
\]
Observe that $P_{(2)}^{\infty}\cap\mathcal{P}=\mathcal{P}_{(2)}$ and
$P_{(1,1)}^{\infty}\cap\mathcal{P}=\mathcal{P}_{(1,1)}$. We have also%
\[
P_{\boxplus}^{\infty}=P_{(2)}^{\infty}\cap P_{(1,1)}^{\infty}=%
{\textstyle\bigoplus\limits_{i\geq1}}
2\mathbb{Z\omega}_{2i}\;\;\text{ and }\;\;\mathcal{P}^{\boxplus}=\mathcal{P}\cap
P_{(2)}^{\infty}\cap P_{(1,1)}^{\infty}\text{.}%
\]
We define the order $\leq_{\boxplus}$ on $\mathcal{P}$ by
\[
\lambda\leq_{\boxplus}\mu\Longleftrightarrow\mu-\lambda\in\mathcal{P}%
^{\boxplus}\text{.}%
\]

\subsection{A combinatorial description of the series $K_{\lambda,0}^{C_\infty}(t)$}

\label{subsec_Cinfinite}

\begin{definition}\label{def6.2}
Consider a partition $\mu$. A vertex $b\in B(\lambda)$ is called $\mu
$-distinguished if there exists $(\nu,\delta)\in\mathcal{P}^{(2)}%
\times\mathcal{P}^{(1,1)}$ such that
\[
\boldsymbol{\varphi}(b)=\nu-\mu\;\;\text{ and }\;\;\boldsymbol{\varepsilon}%
(b)=\delta-\mu.
\]
\end{definition}

\begin{definition}\label{def6.3}
Let $D(\lambda)$ be the set of all vertices in $B(\lambda)$ which are $\mu
$-distinguished for at least a partition $\mu$.
\end{definition}

Clearly, if $b$ is $\mu$-distinguished, then $b$ is $(\mu+\kappa
)$-distinguished for any $\kappa\in\mathcal{P}^{\boxplus}$ (change
$(\nu,\delta)\in\mathcal{P}^{(2)}\times\mathcal{P}^{(1,1)}$ to $(\nu
+\kappa,\delta+\kappa)\in\mathcal{P}^{(2)}\times\mathcal{P}^{(1,1)}$). For any
$b\in D(\lambda)$, set
\[
S_{b}=\{\mu\in\mathcal{P}\mid b\text{ is }\mu\text{-distinguished}\}\,.
\]

\begin{lemma}\label{lem6.4}
The set $S_{b}$ has the form%
\[
S_{b}=\mu_{b}+\mathcal{P}^{\boxplus}\,,%
\]
and $\mu_{b}$ is minimal for $\leq_{\boxplus}$ such that $b$ is $\mu_{b}$-distinguished. Moreover,
for any $\mu\in S_{b}$, we have $\mu_{b}=\mu^{\boxplus}.$
\end{lemma}

\begin{proof}
It suffices to show that $S_{b}$ contains a unique element $\mu_{b}$ minimal
for the order $\leq_{\boxplus}$, since each element $\mu$ of $S_{b}$ can then
be written in the form $\mu=\mu_{b}+\kappa$ with $\kappa$ in $\mathcal{P}%
^{\boxplus}$. So consider $\mu$ and $\mu^{\prime}$ two elements in $S_{b}$
minimal for $\leq_{\boxplus}$ in $S_{b}$.\ Write
\[
\mu=\sum_{i}a_{i}\omega_{i}\;\;\text{ and }\;\;\mu^{\prime}=\sum_{i}a_{i}^{\prime
}\omega_{i}\,.
\]
Since $\mu$ is minimal in $S_{b}$, we must have $a_{i}\in\{0,1\}$ for any even
$i$.\ Indeed, if $a_{i}\geq2$ for an even integer $i\geq1$, we could consider $\mu^{\flat}=\mu-2\omega_{i}\in
\mathcal{P}$. Since $\boldsymbol{\varphi}(b)=\nu-\mu\geq0$, we can consider $\nu^{\flat}=\nu-2\omega_{i}\in\mathcal{P}%
_{(2)}$. Similarly, we have $\boldsymbol{\varepsilon}(b)=\delta-\mu\geq0$, so 
$\delta^{\flat}=\delta-2\omega_{i}\in\mathcal{P}_{(1,1)}$ (as $i$ is even).
Finally we obtain a contradiction since%
\[
\boldsymbol{\varphi}(b)=\nu^{\flat}-\mu^{\flat}\;\;\text{ and }\;\;%
\boldsymbol{\varepsilon}(b)=\delta^{\flat}-\mu^{\flat}\,,%
\]
so $\mu^{\flat}<_{\boxplus}\mu$ belongs to $S_{b}$.

We prove similarly that $a_{i}^{\prime}\in\{0,1\}$ for any even $i$. Now we
can use that $\mu$ and $\mu^{\prime}$ belong to $S_{b}$, which implies that
\[
\mu=\mu^{\prime}\operatorname{mod}P_{(2)}^{+\infty}\;\;\text{ and }\;\;\mu=\mu
^{\prime}\operatorname{mod}P_{(1,1)}^{+\infty}\,.
\]
Since $P_{(2)}^{+\infty}\cap P_{(1,1)}^{+\infty}=P_{\boxplus}^{+\infty}%
=\sum_{i\text{ even}}2\mathbb{Z\omega}_{i}$, we get in fact
\[
\mu=\mu^{\prime}\operatorname{mod}P_{\boxplus}^{+\infty}\text{.}%
\]
This imposes that $a_{i}=a_{i}^{\prime}$ for any odd $i$ and $a_{i}%
=a_{i}^{\prime}\;\mathrm{mod}\;2$ for any even $i$. But we have seen that for any
even $i$, both $a_{i}$ and $a_{i}^{\prime}$ belong to $\{0,1\}$. So we obtain
finally that $a_{i}=a_{i}^{\prime}$ for any $i$ even also. This permits to
conclude that $\mu=\mu^{\prime}$ and $S_{b}$ admits a unique minimal element
for $\leq_{\boxplus}$.
\end{proof}

\bigskip
The following proposition makes more explicit the structure of the distinguished tableaux.
\begin{proposition}
\label{Prop_Char_Distin}Let $b$ be a vertex of $B(\lambda)$ with $\lambda
\in\mathcal{P}$. Then $b$ is distinguished if and only if

\begin{enumerate}
\item $\varepsilon_{i}(b)=0$ for any odd $i,$

\item $\varphi_{i}(b)$ is even for any odd $i.$
\end{enumerate}
Moreover, we then have $\mu_{b}=\sum_{i}(\varphi_{2i}(b)\;\mathrm{mod}\;2)\omega_{2i}%
=:\varphi(b)\;\mathrm{mod}\;2$\,.

\end{proposition}

\begin{proof}
If $\varepsilon_{i}(b)=0$ for any odd $i,$ then $\mathbf{\varepsilon}(b)$
belongs to $\mathcal{P}^{(1,1)}$, and thus $\mathbf{\varepsilon}(b)+\mu$ belongs to
$\mathcal{P}^{(1,1)}$ for any $\mu$ in $\mathcal{P}^{(1,1)}$. Since
$\boldsymbol{\varphi}(b)$ is even for any odd $i$, we will have that
$\boldsymbol{\varphi}(b)+\mu$ belongs to $\mathcal{P}^{(2)}$ for any $\mu$ in
$\mathcal{P}^{(1,1)}$ such that the coefficients of $\omega_{i}$ with $i$ even in the expansions of $\mu$ and
$\boldsymbol{\varphi}(b)$ have the same
parity. Finally $b$ is $\mu$-distinguished for any such $\mu$.

Conversely, assume there exists $\mu$ in $\mathcal{P}$ such that
$\mathbf{\varepsilon}(b)+\mu\in\mathcal{P}^{(1,1)}$ and $\varphi(b)+\mu
\in\mathcal{P}^{(2)}$. Since the coefficients of $\omega_{i}$ with $i$ odd  in the expansion of $\mathbf{\varepsilon}(b)+\mu$
are equal to $0$, and both $\mathbf{\varepsilon}(b)$ and $\mu$ are dominant
weights, we must have that they belong in fact
to $\mathcal{P}^{(1,1)}$. Therefore, the condition $\varphi(b)+\mu
\in\mathcal{P}^{(2)}$ implies that $\varphi_{i}(b)$ is even for any odd $i$.

To determine $\mu_{b}$, we have to choose $\mu$ minimal for the order
$<_{\boxplus}$.\ Since $\varepsilon_{i}(b)=0$ for any odd $i,$ we have in fact
to choose $\mu$ minimal for the order $<_{\boxplus}$ so that $\varphi
(b)+\mu\in\mathcal{P}^{(2)}$.\ This imposes that $\mu_{b}=\sum_{i}%
(\varphi_{2i}(b)\;\mathrm{mod}\;2)\omega_{2i}$.
\end{proof}

\begin{proposition}\label{prop6.6}
We have
\[
\sum_{\nu\in\mathcal{P}^{(2)}}\sum_{\delta\in\mathcal{P}^{(1,1)}}t^{\left\vert
\nu\right\vert /2}c_{\lambda,\delta}^{\nu}=\sum_{b\in D(\lambda)}t^{\left\vert
\boldsymbol{\varphi}(b)+\mu_{b}\right\vert /2}\sum_{\kappa\in\mathcal{P}%
^{\boxplus}}t^{\left\vert \kappa\right\vert /2}\,.
\]
\end{proposition}

\begin{proof}
Recall that $c_{\lambda,\delta}^{\nu}=\mathrm{card}\{b\in B(\lambda
)\mid\boldsymbol{\varepsilon}(b)\leq\delta$ and $\boldsymbol{\varphi
}(b)=\boldsymbol{\varepsilon}(b)+\nu-\delta\}$. For a fixed $b\in B(\lambda)$,
the idea is to gather all the pairs $(\nu,\delta)\in\mathcal{P}^{(2)}%
\times\mathcal{P}^{(1,1)}$ such that $b_{\delta}\otimes b$ is of highest
weight $\nu$. This is equivalent to saying that $b$ is $\mu$-distinguished with
$\mu=\nu-\boldsymbol{\varphi}(b)=\delta-\boldsymbol{\varepsilon}(b)$.\ So we
get%
\[
\sum_{\nu\in\mathcal{P}^{(2)}}\sum_{\delta\in\mathcal{P}^{(1,1)}}t^{\left\vert
\nu\right\vert /2}c_{\lambda,\delta}^{\nu}=\sum_{b\in D(\lambda)}\sum_{\mu\in
S_{b}}t^{\left\vert \boldsymbol{\varphi}(b)+\mu\right\vert /2}\,.
\]
Now by Lemma \ref{lem6.4}, we can write $\boldsymbol{\varphi}(b)+\mu
=\boldsymbol{\varphi}(b)+\mu_{b}+\kappa$, where $\kappa=\mu-\mu_{b}$ belongs to
$\mathcal{P}^{\boxplus}$. This gives%
\[
\sum_{\nu\in\mathcal{P}^{(2)}}\sum_{\delta\in\mathcal{P}^{(1,1)}}t^{\left\vert
\nu\right\vert /2}c_{\lambda,\delta}^{\nu}=\sum_{b\in D(\lambda)}t^{\left\vert
\boldsymbol{\varphi}(b)+\mu_{b}\right\vert /2}\sum_{\kappa\in\mathcal{P}%
^{\boxplus}}t^{\left\vert \kappa\right\vert /2}\,.
\]

\end{proof}

\begin{theorem}\label{thm6.7}
We have
\[
K_{\lambda,0}^{C_\infty}(t)=\sum_{b\in D(\lambda)}t^{\left\vert \boldsymbol{\varphi
}(b)+\mu_{b}\right\vert /2}\,.
\]
\end{theorem}

\begin{proof}
It suffices to observe that
\[
\sum_{\kappa\in\mathcal{P}^{\boxplus}}t^{\left\vert \kappa\right\vert
/2}=\frac{1}{\prod_{i=1}^{\infty}(1-t^{2i})}\,.
\]

\end{proof}

\bigskip

We can get similarly a multivariable version.\ For any $b\in D(\lambda)$, set
\[
\boldsymbol{\varphi}(b)+\mu_{b}=\sum_{i}2a_{i}(b)\omega_{i}\in\mathcal{P}%
^{(2)}\,,%
\]
and assign to each fundamental weight $\omega_{i}$ a formal variable $t_{i}$.
The decomposition
\[
\nu=\boldsymbol{\varphi}(b)+\mu_{b}+\kappa\text{ with }\kappa\in
\mathcal{P}^{\boxplus}%
\]
will give the multivariable version. First let $\boldsymbol{t}=(t_{1}%
,t_{2},\ldots t_{n},\ldots)$ be the sequence of formal variables $t_{i}%
,i\geq1$.\ If one prefers, one can also consider each $t_{i}$ as a real number in
$[0,a]$ with $a<1$. For any $\beta\in P^{+\infty}$ such that $\beta=\sum
_{i}\beta_{i}\omega_{i}$, set $\boldsymbol{t}^{\beta}=\prod_{i\geq1}%
t_{i}^{\beta_{i}}$.

\begin{theorem}\label{thm6.8}
Define the multivariable formal series $K_{\lambda,0}^{C_\infty}(\boldsymbol{t})$ by%
\[
\frac{K_{\lambda,0}^{C_\infty}(\boldsymbol{t})}{\prod_{i=1}^{\infty}(1-t_{2i})}%
=\sum_{\nu\in\mathcal{P}^{(2)}}\sum_{\delta\in\mathcal{P}^{(1,1)}%
}\boldsymbol{t}^{\frac{1}{2}\nu}c_{\lambda,\delta}^{\nu}\,.
\]
Then we have
\[
K_{\lambda,0}^{C_\infty}(\boldsymbol{t})=\sum_{b\in D(\lambda)}\boldsymbol{t}^{\frac
{1}{2}(\boldsymbol{\varphi}(b)+\mu_{b})}\,.
\]
\end{theorem}

\begin{remark}{\rm 
Multivariable generalized exponents defined via the Joseph-Letzter filtration
already appear in the literature (see \cite{Cald}).}
\end{remark}

\subsection{Distinguished tableaux and zero weight King type tableaux}\label{dist-tab}

We are now going to explain how the distinguished tableaux we introduced previously to describe the stable generalized exponents are in
natural bijection with zero weight tableaux very close to King tableaux. We
will in fact consider the sets $T_{C_{\infty}}(\lambda)$ of semistandard
tableaux of shape $\lambda$ on the infinite ordered alphabet $\{1<\overline
{1}<2<\overline{2}<\cdots\}$. There will be no condition on the position of
the barred letters here, contrary to the definition of King tableaux.

We start by discussing the structure of the distinguished tableaux. 
Recall the notation of Section~\ref{subsec_Cinfinite}. For any distinguished
vertex $b$ in $D(\lambda)$, set
\[
\boldsymbol{\theta}(b)=\boldsymbol{\varphi}(b)+\mu_{b}\,,%
\]
and let $\theta_{j}(b)$ be the coefficient of $\omega_j$ in the expansion of $\boldsymbol{\theta}(b)$. Since $\boldsymbol{\theta}(b)$ is a dominant weight for
$\mathfrak{sl}_{\infty}$, it can be regarded as a partition. Recall also that
$\left\vert \lambda\right\vert$ is even, says $\left\vert \lambda\right\vert =2\ell$. In the sequel of this section, we
shall assume that $B(\lambda)$ is realized as the set of semistandard tableaux on
the infinite ordered alphabet $\mathbb{Z}_{>0}$.\ For any integer $i\geq1$, a
reverse lattice skew tableau on $\{2i-1,2i\}$ is a semistandard filling of a
skew Young diagram with columns of height at most $2$ by letters $2i-1$ and
$2i$ whose Japanese reading is a lattice word (i.e., in each left
factor the number of letters $2i$ is less or equal to that of letters $2i-1$).

\begin{example}{\rm 
Assume $i=2$. Then
\begin{equation}%
\begin{tabular}
[c]{lllllllllllllll}\cline{12-12}\cline{12-15}
&  &  &  &  &  &  &  &  &  &  & \multicolumn{1}{|l}{$3$} &
\multicolumn{1}{|l}{$3$} & \multicolumn{1}{|l}{$3$} & \multicolumn{1}{|l|}{$3$%
}\\\cline{8-12}\cline{8-15}
&  &  &  &  &  &  & \multicolumn{1}{|l}{$3$} & \multicolumn{1}{|l}{$3$} &
\multicolumn{1}{|l}{$4$} & \multicolumn{1}{|l}{$4$} & \multicolumn{1}{|l}{$4$}
& \multicolumn{1}{|l}{} &  & \\\cline{4-4}\cline{4-6}\cline{8-12}
&  &  & \multicolumn{1}{|l}{$3$} & \multicolumn{1}{|l}{$3$} &
\multicolumn{1}{|l}{$4$} & \multicolumn{1}{|l}{} &  &  &  &  &  &  &  &
\\\cline{1-4}\cline{1-6}%
\multicolumn{1}{|l}{$3$} & \multicolumn{1}{|l}{$3$} & \multicolumn{1}{|l}{$4$}
& \multicolumn{1}{|l}{$4$} & \multicolumn{1}{|l}{} &  &  &  &  &  &  &  &  &
& \\\cline{1-4}%
\end{tabular}
\label{skewtab}%
\end{equation}
is a reverse lattice skew tableau on $\{3,4\}$.}
\end{example}

The following proposition is a reformulation of Proposition
\ref{Prop_Char_Distin}.

\begin{proposition}
\label{Prop_Char_Distin_T}A semistandard tableau $T$ of shape $\lambda$ is
distinguished if and only if for any integer $i\geq1$, the skew tableau
obtained by keeping only the letters $2i-1$ and $2i$ in $T$ is a reverse
lattice tableau, and the rows of $\boldsymbol{\theta}(T)$ have even lengths.
\end{proposition}

We now explain the correspondence between distinguished
tableaux and zero weight King type tableaux. 

Observe that a tableau $T\ $in $T_{C_{\infty}}(\lambda)$ of weight zero is a
juxtaposition of skew tableaux of weight $0$ on $\{i,\overline{\imath}\}$ obtained
by keeping only the letters $i$ and $\overline{\imath}$. So to obtain a bijection
between the set of distinguished tableaux of shape $\lambda$ and the subset 
$T_{C_{\infty}}^{0}(\lambda)\subset T_{C_{\infty}}(\lambda)$ of zero weight
tableaux, it suffices to describe a bijection between the set of reverse
lattice tableaux on $\{2i-1,2i\}$ of given shape and weight in $2\omega
_{i}\mathbb{Z}_{\geq0},$ and the set of skew tableaux on $\{i,\overline{\imath}\}$
with weight $0$. Now recall that we have the structure of a $U_{q}%
(\mathfrak{sl}_{2})$-crystal on the set of all skew semistandard tableaux of
fixed skew shape both on $\{2i-1,2i\}$ and $\{i,\overline{\imath}\}$. By replacing
each letter $2i-1$ by $i$ and each letter $2i$ by $\overline{\imath}$, we get a
crystal isomorphism $f$.\ The distinguished tableaux correspond to the highest
weight vertices of weight in $2\omega_{i}\mathbb{Z}_{\geq0}$ for the
$\{2i-1,2i\}$-structure, whereas the tableaux of weight $0$ give the vertices
of weight $0$ in the $\{i,\overline{\imath}\}$-crystal structure.\ By observing
that only $U_{q}(\mathfrak{sl}_{2})$-crystals with highest weight in
$2\omega_{i}\mathbb{Z}_{\geq0}$ admit a vertex of weight $0$, which is then
unique, we obtain that the map $\mathcal{C}$ which associates to each zero
weight vertex in the $\{i,\overline{\imath}\}$-crystal structure its highest weight
vertex in the $\{2i-1,2i\}$-crystal structure is the bijection we need. More
precisely, the map $\mathcal{C}$ (resp. its inverse) is obtained as usual: we start by encoding in the
reading of each $\{i,\overline{\imath}\}$-tableau (resp. of each $\{2i-1,2i\}$%
-tableau) the letters $i$ by $+$ and the letters $\overline{\imath}$ by $-$ (resp.
the letters $2i-1$ by $+$ and the letters $2i-1$ by $-$), and next by recursively
deleting all the factors $+-$, thus obtaining a reduced word of the form $-^{m}%
+^{m}$ (resp. $+^{2m}$).\ It then suffices to change the $m$ letters $\overline{\imath}$
corresponding to the $m$ surviving symbols $-$ into $i$ and to apply the
isomorphism $f^{-1}$ (resp. change $m$ letters $2i-1$
corresponding to the rightmost $m$ surviving symbols $+$ into $2i$ and apply the
isomorphism $f$).

\begin{example}\cellsize=3.1ex
{\rm 
The skew tableau of weight $0$ on $\{2,\overline{2}\}$ corresponding to
\rm{(\ref{skewtab})} is
\[%
\tableau{&&&&&&&&&&&2&2&2&\overline{2}\\
&&&&&&&2&\overline{2}&\overline{2}&\overline{2}&\overline{2}\\
&&&2&2&\overline{2}\\
2&2&\overline{2}&\overline{2}}
\]}
\cellsize=2.5ex
\end{example}

In the sequel, we shall abuse the notation and identify the two crystal
structures corresponding up to the isomorphism $f$.

\begin{remarks}\label{rem11.5}{\rm 
(1) It seems not immediate to read $\boldsymbol{\theta}$ directly on zero
weight tableaux.\ The simplest way to do this is to start from a tableau $T\in
T_{C_{\infty}}^{0}(\lambda)$ and compute its associated highest weight tableau
$H(T)$ for the $U_{q}(\mathfrak{sl}_{2}\oplus\cdots\oplus\mathfrak{sl}_{2}%
)$-structure obtained by considering only the action of the crystal operators
indexed by odd integers. So we get%
\[
K_{\lambda,0}^{C_{\infty}}(\boldsymbol{t})=\sum_{T\in T_{C_{\infty}}%
^{0}(\lambda)}\boldsymbol{t}^{\frac{\left\vert \boldsymbol{\theta
}(H(T))\right\vert }{2}}\,.
\]

(2) Let $K_{C_{\infty}}(\lambda)$ be the set of King tableaux on the
infinite ordered alphabet $\{1<\overline{1}<2<\overline{2}<\cdots\}$. Recall
that $T\in T_{C_{\infty}}(\lambda)$ belongs to $K_{C_{\infty}}(\lambda)$ when, 
for any $i=1,\ldots,n$, the letters in row $i$ are greater or equal to $i$.
Since the number of barred letters can only decrease when we compute $H(T)$,
the tableaux $T$ and $H(T)$ either both belong to $K_{C_{\infty}}(\lambda)$ or
belong to $T_{C_{\infty}}(\lambda)\setminus K_{C_{\infty}}(\lambda)$.
Nevertheless, the set $K_{C_{\infty}}^{0}(\lambda)$ of King tableaux of type
$C_{\infty}$ and zero weight is only strictly contained in $T_{C_{\infty}}%
^{0}(\lambda)$ due to the constraints on the rows. In particular, we have%
\[
K_{\lambda,0}^{C_{\infty}}(\boldsymbol{t})\neq\sum_{T\in K_{C_{\infty}}%
^{0}(\lambda)}\boldsymbol{t}^{\frac{1}{2}\boldsymbol{\theta}(H(T))}%
\]
in general, and the finite rank $t$-analogue thus cannot be obtained from the
statistic $\boldsymbol{\theta}$ and King tableaux of zero weight and type
$C_{n}$.
}
\end{remarks}

\begin{example}\label{ex-11}\cellsize=3.1ex
{\rm 
Assume $\lambda=(1,1)$. Then we get
\[
T_{C_{\infty}}^0(\lambda)=\left\{\tableau{k\\ \overline{k}}
\mid k\in\mathbb{Z}_{\geq1}\right\}  \;\;\text{ and }\;\;K_{C_{\infty}}^0(\lambda
)=\left\{\tableau{k\\ \overline{k}}
\mid k\in\mathbb{Z}_{\geq2}\right\}  .
\]
This gives%
\[
H\left(\tableau{k\\ \overline{k}}
\right)  =%
\begin{tabular}
[c]{|c|}\hline
$2k-1$\\\hline
$2k$\\\hline
\end{tabular}
\;\;\text{ and }\;\;\boldsymbol{\varphi}\left(
\begin{tabular}
[c]{|c|}\hline
$2k-1$\\\hline
$2k$\\\hline
\end{tabular}
\right)  =\omega_{2k}\text{ for any }k\geq1.
\]
Therefore
\[
\boldsymbol{\theta}\left(
\begin{tabular}
[c]{|c|}\hline
$2k-1$\\\hline
$2k$\\\hline
\end{tabular}
\right)  =2\omega_{2k}\text{ for any }k\geq1\text{.}%
\]
Finally%
\[
K_{\lambda,0}^{C_{\infty}}(\boldsymbol{t})=\sum_{k\geq1}t_{2k}\;\;\text{ and }\;\;K_{\lambda,0}^{C_{\infty}}(t)=\sum_{k\geq1}t^{2k}=\frac{t^{2}}{1-t^{2}}\,.
\]
\cellsize=2.5ex
}
\end{example}

\section{Type $C_{n}$ generalized exponents via the Sundaram LR tableaux}\label{sundaram-lr}

\subsection{Sundaram description of the coefficients $c_{\nu}^{\lambda
}(\mathfrak{sp}_{2n})$}

Recall that in type $C_{n},$ the equality $c_{\nu}^{\lambda}(\mathfrak{sp}%
_{2n})=\sum_{\delta\in\mathcal{P}^{(1,1)}}c_{\lambda,\delta}^{\nu}$ only holds
when $\nu\in\mathcal{P}_{n}$, in which case we have in fact%
\[
c_{\nu}^{\lambda}(\mathfrak{sp}_{2n})=\sum_{\delta\in\mathcal{P}_{n}^{(1,1)}%
}c_{\lambda,\delta}^{\nu}\,.
\]
In the general case of a partition $\nu\in\mathcal{P}_{2n}$, we have by a
result of Sundaram (see \cite[Corollary~3.12]{Sun})
\[
c_{\nu}^{\lambda}(\mathfrak{sp}_{2n})=\sum_{\delta\in\mathcal{P}^{(1,1)}%
}\widehat{c}_{\lambda,\delta}^{\nu}\,,%
\]
where $\widehat{c}_{\lambda,\delta}^{\nu}$ is the number of Sundaram-LR
tableaux, that is, the number of LR tableaux of shape $\nu/\lambda$ and weight
$\delta$ filled with letters in $\{1,\ldots,2n\}$ such that each odd letter
$2i+1$ appears no lower (English convention) than row $(n+i)$ in $\nu$ (the
rows being numbered from top to bottom). Observe that for any partition
$\kappa$ in $\mathcal{P}_{2n}^{\boxplus}$, a Sundaram-LR tableau of shape
$\nu/\lambda$ and weight $\delta$ can be easily turned into a Sundaram-LR
tableau of shape $(\nu+\kappa)/\lambda$ and weight $\delta+\kappa$ by adding
letters $i$ in rows $i$, which does not violate the Sundaram condition.

\subsection{LR-tableaux and crystals}

Given $\nu,\lambda,\mu$ three partitions, the Littlewood-Richardson coefficient $c_{\lambda\mu}^\nu$ 
is equal to the cardinality of the $4$ following sets:

\begin{enumerate}
\item the set of LR tableaux of shape $\nu/\lambda$ and weight $\mu$,

\item the set of LR tableaux of shape $\nu/\mu$ and weight $\lambda$,

\item the set of vertices $b\in B(\lambda)$ such that $\boldsymbol{\varepsilon
}(b)\leq\mu$,

\item the set of vertices $b^{\prime}\in B(\mu)$ such that
$\boldsymbol{\varepsilon}(b^{\prime})\leq\lambda$.
\end{enumerate}

Now there exist bijections between all these sets. Given a LR tableau $\tau$
of shape $\nu/\mu$ and weight $\lambda$, we obtain the corresponding tableau
$\mathbf{T}(\tau)$ in $B(\lambda)$, called companion tableau, by placing in the $k$-th row of the Young diagram $\lambda$
the numbers of the rows of $\tau$ containing an entry $k$. 

\begin{example}{\rm 
For
\[
\tau=%
\begin{tabular}
[c]{lll|ll}\cline{4-5}
&  &  & $1$ & \multicolumn{1}{|l|}{$1$}\\\cline{3-5}
&  & \multicolumn{1}{|l|}{$1$} & $2$ & \multicolumn{1}{|l|}{$2$}%
\\\cline{2-3}\cline{2-5}
& \multicolumn{1}{|l}{$1$} & \multicolumn{1}{|l|}{$3$} &  & \\\cline{2-3}
& \multicolumn{1}{|l}{$2$} & \multicolumn{1}{|l|}{$4$} &  & \\\cline{1-3}%
\multicolumn{1}{|l}{$1$} & \multicolumn{1}{|l}{$3$} & \multicolumn{1}{|l|}{$5$%
} &  & \\\cline{1-3}%
\end{tabular}
\ \text{ we get }\mathbf{T}(\tau)=%
\begin{tabular}
[c]{|l|llll}\hline
$1$ & $1$ & \multicolumn{1}{|l}{$2$} & \multicolumn{1}{|l}{$3$} &
\multicolumn{1}{|l|}{$5$}\\\hline
$2$ & $2$ & \multicolumn{1}{|l}{$4$} & \multicolumn{1}{|l}{} & \\\cline{1-3}%
$3$ & $5$ & \multicolumn{1}{|l}{} &  & \\\cline{1-2}%
$4$ &  &  &  & \\\cline{1-1}%
$5$ &  &  &  & \\\cline{1-1}%
\end{tabular}\,.
\]
}
\end{example}

Now we can proceed as in Section~\ref{Sec_Cinfinite} by first determining the
subset of $\widehat{D}(\lambda)\subset B^{\mathfrak{gl}_{2n}}(\lambda)$ coming
from Sundaram-LR tableaux ($D(\lambda)$ would correspond to all the LR
tableaux as in the previous section). To do this we proceed as
follows.

\begin{enumerate}
\item Start with a Sundaram-LR tableau of shape $\nu/\lambda$ and weight
$\delta$, and determine its associated tableau $\mathbf{T}(\tau)$ of shape
$\delta$ and entries in $\{1,\ldots,2n\}$.

\item Observe that $T_{\lambda}\otimes\mathbf{T}(\tau)$ is of highest weight
$\nu$ in $B(\lambda)\otimes B(\delta)$.

\item Compute the combinatorial $R$-matrix, and obtain $\widehat{\mathbf{T}}(\tau)$ in
$B(\lambda)$ such that $T_{\lambda}\otimes\mathbf{T}(\tau)\leftrightarrows
T_{\delta}\otimes\widehat{\mathbf{T}}(\tau)$. Here we can choose
the version of the combinatorial $R$-matrix given by the Henriques-Kamnitzer commutor \cite{hakccc,katccd}, which has several concrete realizations; see Section~\ref{sunkwo} for more details.

\item Finally, define $\widehat{D}(\lambda)$ as the subset of tableaux $T\in
D(\lambda)$ for which there exists $(\nu,\delta)\in\mathcal{P}_{2n}%
^{(2)}\times\mathcal{P}_{2n}^{(1,1)}$ and $\tau$ a Sundaram-LR tableau of
shape $\nu/\lambda$ and weight $\delta$ such that $T=\widehat{\mathbf{T}}%
(\tau)$.
\end{enumerate}

Now, we have
\[
\sum_{\kappa\in\mathcal{P}_{2n}^{\boxplus}}t^{\left\vert \kappa\right\vert
/2}=\frac{1}{\prod_{i=1}^{n}(1-t^{2i})}\,,%
\]
since $\mathcal{P}_{2n}^{\boxplus}$ is obtained by dilating by a factor $2$
the set $\mathcal{P}_{n}$ (i.e. each square becomes a $\boxplus$). By using
similar arguments\footnote{Here, we need to use that for any $\kappa
\in\mathcal{P}_{\boxplus}^{(2n)}$, one can produce a Sundaram-LR tableau of
shape $(\nu+\kappa)/\lambda$ and weight $\delta+\kappa$ starting from any
Sundaram-LR tableau of shape $\nu/\lambda$ and weight $\delta$.} to those of
Section~\ref{Sec_Cinfinite}, we obtain the following result.

\begin{theorem}
We have
\[
K_{\lambda,0}^{C_{n}}(t)=\sum_{b\in\widehat{D}(\lambda)}t^{\left\vert
\boldsymbol{\varphi}(b)+\mu_{b}\right\vert /2}\,.
\]

\end{theorem}

\section{Type $C_n$ generalized exponents via the Kwon model}\label{genexpc}

In this section, we refine the results in Sections~\ref{Sec_Cinfinite} and \ref{dist-tab} to the finite type $C_n$, based on Kwon's model for the corresponding branching coefficients \cite{kwoces,kwoldk}. We also need to use a combinatorial map realizing the conjugation symmetry of Littlewood-Richardson coefficients. It turns out that Kwon's model, the version of the conjugation symmetry map used here, and the distinguished tableaux in Section~\ref{subsec_Cinfinite} fit together in a beautiful way. This allows us to express the related statistic in terms of a natural combinatorial labeling of the vertices of weight 0 in the corresponding type $C_n$ crystal of highest weight $\lambda$, namely the corresponding tableaux due to King \cite{King}. In this way, we obtain a more explicit result than the one in Section~\ref{sundaram-lr} in terms of Littlewood-Richardson-Sundaram tableaux. 

\subsection{The Littlewood-Richardson conjugation symmetry}\label{lrcs} 

Consider partitions $\lambda\in {\mathcal P}_n$ and $\delta,\nu\in{\mathcal P}_m$ with $n\le m$. We will exhibit combinatorially the equality of Littlewood-Richardson (LR) coefficients $c_{\lambda,\delta}^\nu=c_{\lambda',\delta'}^{\nu'}$. Throughout, we denote by $\delta^{\rm rev}$ the reverse of $\delta$, namely $\delta^{\rm rev}=(\delta_1^{\rm rev}\le\ldots\le \delta_m^{\rm rev})$, where we add leading 0's if necessary. 

Let ${\rm LR}_{\lambda,\delta}^\nu$ denote the set of Littlewood-Richardson tableaux $T$ of shape $\lambda$ and content $\nu/\delta$; in other words, $T\in B_m(\lambda)$ and $H_\delta\otimes T$ is a highest weight element of weight $\nu$, where $H_\delta$ denotes the Yamanouchi tableau of shape $\delta$. We will construct a bijection $T\mapsto T'$ between ${\rm LR}_{\lambda,\delta}^\nu$  and ${\rm LR}_{\lambda',\delta'}^{\nu'}$, where $T'$ is viewed as an element of $B_{\ell(\nu')}(\lambda')$. The construction has the following three steps.

{\em Step {\rm 1}.} Apply the Sch\"utzenberger evacuation \cite{fulyt} (realizing the Lusztig involution) to $T$ within the crystal $B_m(\lambda)$, and obtain $S(T)\in B_m(\lambda)$. 

{\em Step {\rm 2}.} Transpose the tableau $S(T)$ and denote the resulting filling of $\lambda'$ by $S(T)^{\rm{tr}}$.

{\em Step {\rm 3}.} For each $i=1,\ldots,m$, consider in $S(T)^{\rm{tr}}$ the vertical strip of $i$'s, and replace these entries, scanned from northeast to southwest, with $\delta_i^{\rm rev}+1,\,\delta_i^{\rm rev}+2,\,\ldots$, respectively. 

\begin{example}{\rm Let $n=3$, $m=4$, $\lambda=(4,3,1)$, $\nu=(5,4,4,2)$, $\delta=(3,3,1)$, and $\delta^{\rm rev}=(0,1,3,3)$. Consider the following tableau of shape $\nu/\delta$ and content $\lambda$ whose reverse row word is a lattice permutation, and its associated companion tableau $T\in B_4(\lambda)$:
\[\tableau{{}&{}&{}&1&1\\{}&{}&{}&2\\{}&1&2&3\\1&2}\,,\;\;\;\;\;\;\;T=\tableau{1&1&3&4\\2&3&4\\3}\,.\]
The tableau $S(T)\in B_4(\lambda)$ and $S(T)^{\rm{tr}}$ are
\[S(T)=\tableau{1&1&2&2\\2&3&4\\4}\,,\;\;\;\;\;\;\;S(T)^{\rm{tr}}=\tableau{1&2&4\\1&3\\2&4\\2}\,.\]
Step 3 above produces 
\[T'=\tableau{1&2&4\\2&4\\3&5\\4}\]
in $B_5(\lambda')$. One can then check that the same procedure maps $T'$ back to $T$.}
\end{example}

\begin{theorem}\label{lrconjsym} The above map $T\mapsto T'$ is a bijection between ${\rm LR}_{\lambda,\delta}^\nu$  and ${\rm LR}_{\lambda',\delta'}^{\nu'}$.
\end{theorem}

\begin{proof}
A bijection realizing the conjugation symmetry of the LR coefficients was given on the skew LR tableaux (of shape $\nu/\delta$ and content $\lambda$) as the map $\rho_3$ in \cite{acmlte}. It is not hard to show that on the companion tableau it is described by the above algorithm. The key fact involved here is that the crystal action of the longest permutation in $S_m$ on the skew LR tableau corresponds to the Sch\"utzenberger involution applied to the companion tableau. This fact is well-known to experts, and is based on the so-called ``double crystal graph structure'' on biwords \cite{lasdcg}. According to this, the action of crystal operators on words corresponds to jeu de taquin slides on two-row tableaux, where the latter are involved in the construction of the Sch\"utzenberger involution; see also \cite{acmlte,fulyt} for more details.
\end{proof}

\begin{remarks} {\rm (1) It is easy to see that, if we change $m$ in the above construction, Step 1 is different, but the final result is the same.

(2) It was shown in \cite{fulyt} and \cite{acmlte} that the above map coincides with the maps constructed by: Hanlon-Sundaram \cite{hasblr}, White \cite{whiht}, and Benkart-Sottile-Stroomer \cite{bssts}. In fact, Benkart-Sottile-Stroomer also give a characterization of their map based on Knuth and dual Knuth equivalences. Furthermore, the inverse of the conjugation symmetry map is described by the same procedures~\cite{hasblr}.}
\end{remarks}

\subsection{Kwon's model}

In this section we describe Kwon's {\em spin model} for crystals of classical type \cite{kwoces,kwoldk}, which is also used to express certain branching coefficients, and leads to an interesting branching duality. 

We start with the Lie algebra $\mathfrak{sp}_{2n}$, with the corresponding long simple root being indexed by~$0$. Consider a dominant weight $\lambda\in{\mathcal P}_n$, and let $\Lambda^\mathfrak{sp}(\lambda):=n\Lambda_0^\mathfrak{sp}+\lambda_1'\varepsilon_1+\lambda_2'\varepsilon_2+\ldots$, where $\Lambda_0^\mathfrak{sp}$ is the 0-fundamental weight for $\mathfrak{sp}_\infty$. Kwon first constructs a combinatorial model for the crystal $B(\mathfrak{sp}_\infty,\Lambda^\mathfrak{sp}(\lambda))$, which we now briefly describe.

The model is built on a certain family ${\mathbf T}^\mathfrak{sp}(\lambda,n)$ formed by sequences ${\mathbf T}:=C_1C_2\ldots C_{2n}$ of fillings with positive integers of column shapes. These sequences satisfy the following conditions:
\begin{enumerate}
\item each pair $C_{2i-1}C_{2i}$ is a SSYT of shape $(\lambda_i+\delta_{2i-1}^{\rm rev},\delta_{2i}^{\rm rev})'$, denoted $T_i$, where $\delta$ is some partition in ${\mathcal P}_{2n}^{(1,1)}$, which means that $\delta_{2i-1}=\delta_{2i}$ for $i=1,\ldots,n$;
\item each pair $(T_i,T_{i+1})$ satisfies certain compatibility conditions, see \cite[Definition~3.2]{kwoces}.
\end{enumerate}

For each $i\ge 0$, Kwon defines crystal operators $\widetilde{e}_i$, $\widetilde{f}_i$ on the set of pairs of columns described in (1) above, and then extends them to ${\mathbf T}^\mathfrak{sp}(\lambda,n)$ via the usual tensor product rule. With this structure, it is proved that ${\mathbf T}^\mathfrak{sp}(\lambda,n)$  is isomorphic to the crystal  $B(\mathfrak{sp}_\infty,\Lambda^\mathfrak{sp}(\lambda))$.

Following \cite{kwoces}, we introduce further notation related to the above objects. The left and right columns of $T_i$ defined above are denoted $T_i^L$, $T_i^R$, respectively. The bottom part of $T_i^L$ of height $\lambda_i$ is denoted $T_i^{\rm tail}$; the remaining top part together with $T_i^R$, which form a SSYT of rectangular shape $(\delta_{2i-1}^{\rm rev},\delta_{2i}^{\rm rev})'$, is denoted $T_i^{\rm body}$. In the filling ${\mathbf T}$ the columns are arranged such that ${\mathbf T}^{\rm body}:=(T_1^{\rm body},\ldots,T_n^{\rm body})$ is a filling of the shape $(\delta')^\pi$ denoting the rotation of $\delta'$ by $180^\circ$. Kwon also uses the notation ${\mathbf T}^{\rm tail}:=(T_1^{\rm tail},\ldots,T_n^{\rm tail})$, which is a filling of the shape $\lambda'$. As usual, ${\rm content}({\mathbf T})$ is defined as the sequence $(c_1,c_2,\ldots)$, where $c_i$ is the number of entries $i$ in ${\mathbf T}$. We identify ${\mathbf T}$ with its column word, denoted ${\rm word}({\mathbf T})$, which is obtained by reading the columns from right to left and from top to bottom. Let $L({\mathbf T})$ be the maximal length of a weakly decreasing subword of ${\rm word}({\mathbf T})$.

\begin{lemma}\label{lem1} \cite{kwoces} If $L({\mathbf T})\le n$, then we have
\begin{enumerate}
\item ${\mathbf T}^{\rm body}$ is a SSYT of shape $(\delta')^\pi$ for some $\delta\in {\mathcal P}_{2n}^{(1,1)}$, and ${\mathbf T}^{\rm tail}$ is a SSYT of shape $\lambda'$;
\item ${\mathbf T}\equiv {\mathbf T}^{\rm body}\otimes {\mathbf T}^{\rm tail}$, where $\equiv$ denotes the usual (type $A$) plactic equivalence.
\end{enumerate}
\end{lemma}

Now fix a partition $\nu\in{\mathcal P}_{2n}$. Consider the set ${\rm LR}^\lambda_\nu(\mathfrak{sp}_{2n})$ of type $A$ highest weight elements ${\mathbf T}$ in ${\mathbf T}^\mathfrak{sp}(\lambda,n)$ with ${\rm content}({\mathbf T})=\nu'$; in other words, we have $\widetilde{e}_i({\mathbf T})=0$ for all $i>0$. 

\begin{theorem}\label{branching} \cite{kwoces} The cardinality of ${\rm LR}^\lambda_\nu(\mathfrak{sp}_{2n})$ is equal to the branching coefficient $c^\lambda_\nu(\mathfrak{sp}_{2n})$.
\end{theorem} 

Considering ${\mathbf T}$ in ${\rm LR}^\lambda_\nu(\mathfrak{sp}_{2n})$, we have by definition ${\mathbf T}\equiv H_{\nu'}$. Thus, in the special case $\nu\in{\mathcal P}_{n}$, it follows from Lemma~\ref{lem1} that ${\mathbf T}^{\rm body}\equiv {H}_{\delta'}$ and ${\mathbf T}^{\rm tail}\in{\rm LR}_{\lambda',\delta'}^{\nu'}$, for some $\delta\in{\mathcal P}_{2n}^{(1,1)}$. Here and throughout, we use implicitly the 
fact that the crystal operators preserve the plactic equivalence. Based on the above facts, the following result is proved.

\begin{theorem}\label{tailbij} \cite{kwoces} Assume $\nu\in{\mathcal P}_{n}$. The map ${\mathbf T}\mapsto {\mathbf T}^{\rm tail}$ is a bijection 
\[{\rm LR}^\lambda_\nu(\mathfrak{sp}_{2n})\longrightarrow \bigsqcup_{\delta\in{\mathcal P}_{2n}^{(1,1)}}{\rm LR}_{\lambda',\delta'}^{\nu'}\,.\]
\end{theorem}

As $c_{\lambda',\delta'}^{\nu'}=c_{\lambda,\delta}^\nu$, Theorem~\ref{tailbij} gives a simple combinatorial realization of the well-known stable branching rule (for $\nu\in{\mathcal P}_{n}$):
\[c^\lambda_\nu(\mathfrak{sp}_{2n})=\sum_{\delta\in{\mathcal P}_{2n}^{(1,1)}}c_{\lambda,\delta}^{\nu}\,.\]

Without the assumption $L({\mathbf T})\le n$, Lemma~\ref{lem1} fails, i.e., ${\mathbf T}^{\rm body}$ and ${\mathbf T}^{\rm tail}$ are no longer SSYT of the corresponding shapes. Kwon addresses this complication in \cite[Section~5]{kwoldk}, by first mapping ${\mathbf T}=C_1C_2\ldots C_{2n}$ to a new filling $\overline{\mathbf T}$. The construction is based on jeu de taquin on successive columns, which is used to perform the following operations in the indicated order: 
\begin{itemize}
\item move $\lambda_2$ entries from column $C_3$ to the second column;
\item move $\lambda_3$ entries from column $C_5$ to the third column (past the fourth column in-between);
\item continue in this fashion, and end by moving $\lambda_n$ entries from column $C_{2n-1}$ to the $n$-th column (past the columns in-between).
\end{itemize}
It is easy to see that the above operations can always be performed. The shape of the filling $\overline{\mathbf T}$ is a skew Young diagram, obtained by gluing $\lambda'$ to the bottom of $(\delta')^\pi$, such that their first columns are aligned (we view $(\delta')^\pi$ as a diagram with $2n$ columns, where possibly the leading ones have length $0$). The fillings of shapes $\lambda'$ and $(\delta')^\pi$ are denoted $\overline{\mathbf T}^{\rm tail}$ and $\overline{\mathbf T}^{\rm body}$, respectively. We have an analogue of Lemma~\ref{lem1}.

\begin{lemma}\label{lem2} \cite{kwoldk} The following hold:
\begin{enumerate}
\item $\overline{\mathbf T}^{\rm body}$ is a SSYT of shape $(\delta')^\pi$ for some $\delta\in {\mathcal P}_{2n}^{(1,1)}$, and $\overline{\mathbf T}^{\rm tail}$ is a SSYT of shape $\lambda'$;
\item ${\mathbf T}\equiv\overline{\mathbf T}\equiv \overline{\mathbf T}^{\rm body}\otimes \overline{\mathbf T}^{\rm tail}\,.$
\end{enumerate}
\end{lemma}

The difficulty lies in the first part of this lemma, whose proof is highly technical. The second part follows from the first one simply by noting that jeu de taquin is compatible with the plactic equivalence, and that the row and column words of a skew SSYT are placticly equivalent. 

In \cite[Remark~5.6]{kwoldk} it is observed that, if $L({\mathbf T})\le n$ (in particular, if ${\mathbf T}\in{\rm LR}^\lambda_\nu(\mathfrak{sp}_{2n})$ and $\nu\in{\mathcal P}_{n}$), then we have $\overline{\mathbf T}^{\rm body}={\mathbf T}^{\rm body}$ and $\overline{\mathbf T}^{\rm tail}={\mathbf T}^{\rm tail}$, so Lemma~\ref{lem1} is a special case of Lemma~\ref{lem2}. In fact, we can show that the mentioned equalities also hold for the elements of ${\rm LR}^\lambda_\nu(\mathfrak{sp}_{2n})$, for any $\nu\in{\mathcal P}_{2n}$. This leads to the following generalization of Theorem~\ref{tailbij}. To state it, we define $\overline{\rm LR}_{\lambda',\delta'}^{\nu'}$ to be the subset of ${\rm LR}_{\lambda',\delta'}^{\nu'}$ consisting of fillings ${S}$ with the following property: denoting the first row of ${S}$ by $(r_1\le\ldots\le r_p)$, for $p\le n$, we have 
\begin{equation}\label{condbranching}r_i>\delta_{2i-1}^{\rm rev}=\delta_{2i}^{\rm rev}\;\;\;\;\;\mbox{for $i=1,\ldots,p$}\,.\end{equation}
Let $\overline{c}_{\lambda,\delta}^\nu$ be the cardinality of $\overline{\rm LR}_{\lambda',\delta'}^{\nu'}$. 

\begin{theorem}\label{gentail} Consider ${\mathbf T}$ in ${\rm LR}^\lambda_\nu(\mathfrak{sp}_{2n})$, and let $(\delta')^\pi$ be the shape of ${\mathbf T}^{\rm body}$. 

{\rm (1)} We have ${\mathbf T}^{\rm body}\equiv {H}_{\delta'}$ and ${\mathbf T}^{\rm tail}\in \overline{\rm LR}_{\lambda',\delta'}^{\nu'}$.  

{\rm (2)} The map ${\mathbf T}\mapsto {\mathbf T}^{\rm tail}$ is an injection 
\[{\rm LR}^\lambda_\nu(\mathfrak{sp}_{2n})\lhook\joinrel\xrightarrow{\;\;\;\;\;\;} \bigsqcup_{\delta\in{\mathcal P}_{2n}^{(1,1)}}{\rm LR}_{\lambda',\delta'}^{\nu'}\,,\]
and its image is $\bigsqcup_{\delta\in{\mathcal P}_{2n}^{(1,1)}}\overline{\rm LR}_{\lambda',\delta'}^{\nu'}$. 
\end{theorem}

\begin{proof} Consider the filling $\overline{\mathbf T}$ obtained from ${\mathbf T}$ via the procedure described above. Since ${\mathbf T}\equiv H_{\nu'}$, it follows that $\overline{\mathbf T}^{\rm body}\equiv H_{\delta'}$ and $\overline{\mathbf T}^{\rm tail}\in{\rm LR}_{\lambda',\delta'}^{\nu'}$, by Lemma~\ref{lem2}. Thus, the $i$-th column of the SSYT $\overline{\mathbf T}^{\rm body}$ is $(1<2<\ldots<\delta_i^{\rm rev})$, for $i=1,\ldots,2n$. 

The procedure ${\mathbf T}\mapsto\overline{\mathbf T}$, which is based on jeu de taquin on successive columns, is reversible. We claim that this reverse procedure $\overline{\mathbf T}\mapsto{\mathbf T}$ simply slides the columns of $\overline{\mathbf T}^{\rm tail}$ horizontally (that is, restricts to horizontal jeu de taquin moves) from positions $1,\ldots,n$ within $\overline{\mathbf T}$ to positions $1,3,\ldots,2n-1$, respectively, while the columns of $\overline{\mathbf T}^{\rm body}$ do not move (recall that the columns of $\overline{\mathbf T}^{\rm tail}$ and ${\mathbf T}^{\rm tail}$ within $\overline{\mathbf T}$ and ${\mathbf T}$ have their top entries on the same row). This means that $\overline{\mathbf T}^{\rm body}={\mathbf T}^{\rm body}$ and $\overline{\mathbf T}^{\rm tail}={\mathbf T}^{\rm tail}$. Therefore, the map ${\mathbf T}\mapsto {\mathbf T}^{\rm tail}$ is the desired injection. Moreover, the image of this map is contained in $\bigsqcup_{\delta\in{\mathcal P}_{2n}^{(1,1)}}\overline{\rm LR}_{\lambda',\delta'}^{\nu'}$ because the columns of ${\mathbf T}$ are strictly increasing.

The proof of the above claim is based on the following fact. Consider columns $(1<2<\ldots<k<c_1<\ldots<c_s)$ and $(1<2<\ldots<l)$ with $k\le l$, and assume that we can move $s$ entries from the first one to the second one via jeu de taquin. To do this, we start by aligning the two columns such that they form a skew SSYT, and this can be done by placing $k\le l$ in the same row. We claim that $c_1>l$, which implies that the resulting columns are $(1<2<\ldots<k)$ and $(1<2<\ldots<l<c_1<\ldots<c_s)$, as needed. Indeed, if $c_1\le l$, then $k\le l-1$, $k-1\le l-2$, etc., so we can align the two initial columns such that all the mentioned pairs are in the same rows. But then at most $s-1$ entries can move from the first column to the second one, which is a contradiction. 

It remains to prove that any filling ${S}\in\overline{\rm LR}_{\lambda',\delta'}^{\nu'}$ is in the image of the given map. Consider the SSYT whose $i$-th column is $(1<2<\ldots<\delta_i^{\rm rev})$, and glue the columns of ${S}$ to the bottom of the columns of the former in positions $1,3,5,\ldots$. It is easy to check that the resulting filling ${\mathbf T}$ satisfies the conditions in \cite[Definition~3.2]{kwoces}, so ${\mathbf T}\in{\mathbf T}^\mathfrak{sp}(\lambda,n)$. Now observe that the procedure ${\mathbf T}\mapsto\overline{\mathbf T}$ consists of sliding the columns of ${S}$ within ${\mathbf T}$ horizontally, as far left as possible, which means that $\overline{\mathbf T}^{\rm body}={\mathbf T}^{\rm body}$ and $\overline{\mathbf T}^{\rm tail}={\mathbf T}^{\rm tail}$. By Lemma~\ref{lem2}~(2), it follows that ${\mathbf T}\equiv H_{\delta'}\otimes {S}$. The latter
 is a highest weight element, as ${S}\in{\rm LR}_{\lambda',\delta'}^{\nu'}$, and this implies ${\mathbf T}\in{\rm LR}^\lambda_\nu(\mathfrak{sp}_{2n})$.
\end{proof}

By combining Theorems~\ref{branching} and \ref{gentail}, we obtain a simple combinatorial description of the branching coefficient $c^\lambda_\nu(\mathfrak{sp}_{2n})$ in full generality.

\begin{corollary}\label{kwonbrrule} We have
\[c^\lambda_\nu(\mathfrak{sp}_{2n})=\sum_{\delta\in{\mathcal P}_{2n}^{(1,1)}}\overline{c}_{\lambda,\delta}^{\nu}\,.\]
\end{corollary}

\subsection{Generalized exponents in terms of distinguished tableaux}\label{ranknlusz}

The goal is to derive a finite rank analogue of the results in {Section~\ref{subsec_Cinfinite}, that is, for type $C_n$. 

We use the same notation, except that everything now happens in finite rank. Thus, we require $\lambda\in{\mathcal P}_n$. We denote the underlying type $A_{2n-1}$ crystal by $B_{2n}(\lambda)$, and the set of distinguished tableaux contained in it by $D_{2n}(\lambda)$. The latter is defined like in {Definition~\ref{def6.3}, and is characterized by the analogues of the two conditions in {Proposition~\ref{Prop_Char_Distin}}. Recall the set $S_b$, whose analogue is defined for any $b\in B_{2n}(\lambda)$ by
\[S_{b,n}:=\{\mu\in{\mathcal P}_{2n}\,:\,\boldsymbol{\varphi}(b)+\mu\in{\mathcal P}_{2n}^{(2)}\,,\;\,\boldsymbol{\varepsilon}(b)+\mu\in{\mathcal P}_{2n}^{(1,1)}\}\,.\]
Note that the above conditions on $\mu$ simply mean that 
\[b\in\rm{LR}_{\lambda,\delta}^\nu\;\;\;\mbox{for $\delta:=\boldsymbol{\varepsilon}(b)+\mu\in{\mathcal P}_{2n}^{(1,1)}$,  $\;\nu:=\boldsymbol{\varphi}(b)+\mu\in{\mathcal P}_{2n}^{(2)}$}\,.\]

The analogue of the weight $\mu_b$, denoted $\mu_{b,n}$, is constructed as in {Proposition~\ref{Prop_Char_Distin}~(2)}:
\begin{equation}\label{mub}\mu_{b,n}:=\sum_{i=1}^{n-1}(\varphi_{2i}(b)\;\,\mbox{mod $2$})\,\omega_{2i}\,.\end{equation}
With this notation, we have the analogue of {Lemma~\ref{lem6.4}}, namely
\begin{equation}\label{sb}
S_{b,n}:=\casetwoother{\mu_{b,n}+{\mathcal P}_{2n}^\boxplus}{b\in D_{2n}(\lambda)}{\emptyset}
\end{equation}

We also need some new notation. Let $D_{2n}^*(\lambda)$ be defined by ``swapping'' the conditions characterizing $D_{2n}(\lambda)$ in {Proposition~\ref{Prop_Char_Distin}}; namely, $D_{2n}^*(\lambda)$ consists of $b\in B_{2n}(\lambda)$ such that
\begin{enumerate}
\item[(C1)] $\varphi_{i}(b)=0$ for any odd $i$;
\item[(C2)] $\varepsilon_{i}(b)$ is even for any odd $i$.
\end{enumerate}
Let $\overline{D}_{2n}^*(\lambda)$ be the subset of $D_{2n}^*(\lambda)$ consisting of those SSYT satisfying the following flag condition:

$\!$(C3) the entries in row $i$ are at least $2i-1$, for $i=1,\ldots,n$. 

\noindent Finally, we define the analogues of $\boldsymbol{\varepsilon}(b)$, $\boldsymbol{\varphi}(b)$, and of $\mu_{b,n}$ in \eqref{mub} by
\[\boldsymbol{\varepsilon}^*(b):=\sum_{i=1}^{2n-1}{\varepsilon}_{2n-i}(b)\,\omega_i\,,\;\;\;\;\boldsymbol{\varphi}^*(b):=\sum_{i=1}^{2n-1}{\varphi}_{2n-i}(b)\,\omega_i\,,\;\;\;\;\mu_{b,n}^*:=\sum_{i=1}^{n-1}(\varepsilon_{2n-2i}(b)\;\,\mbox{mod $2$})\,\omega_{2i}\,.\]

Now recall Lusztig's involution $S$ on the crystal $B_{2n}(\lambda)$. This is realized by Sch\"utzenberger's evacuation \cite{fulyt}, and is known to commute with the crystal operators as follows: 
\begin{equation}\label{sef}\widetilde{e}_i S=S\widetilde{f}_{2n-i}\,,\;\;\;\;\;\widetilde{f}_i S=S\widetilde{e}_{2n-i}\,.\end{equation}
It is then clear that $S$ maps $D_{2n}(\lambda)$ to $D_{2n}^*(\lambda)$. It also follows that we have
\begin{equation}\label{sepsphi}
\varepsilon_i(S(b))=\varphi_{2n-i}(b)\,,\;\;\;\;\;\varphi_i(S(b))=\varepsilon_{2n-i}(b)\,,
\end{equation}
and therefore
\begin{equation}\label{epsmubstar} \boldsymbol{\varepsilon}(b)=\boldsymbol{\varphi}^*(S(b))\,,\;\;\;\;\;\boldsymbol{\varphi}(b)=\boldsymbol{\varepsilon}^*(S(b))\,,\;\;\;\;\;\mu_{b,n}=\mu_{S(b),n}^*\,.
\end{equation}

We start with the analogue of {Proposition~\ref{prop6.6}}.

\begin{theorem}\label{thmsums} We have
\[
\sum_{\nu\in\mathcal{P}_{2n}^{(2)}}\sum_{\delta\in\mathcal{P}_{2n}^{(1,1)}}t^{\left\vert
\nu\right\vert /2}\,\overline{c}_{\lambda,\delta}^{\nu}=\sum_{b\in \overline{D}_{2n}^*(\lambda)}t^{\left\vert
\boldsymbol{\varepsilon}^*(b)+\mu_{b,n}^*\right\vert /2}\sum_{\kappa\in\mathcal{P}_{2n}^{\boxplus}}t^{\left\vert \kappa\right\vert /2}\,.
\]
\end{theorem}

The proof of this theorem is based on the following lemma. To state it, let us recall the Littlewood-Richardson conjugation symmetry map in Section~\ref{lrcs}. Following the notation used there, we set $m=2n$, and given fixed $\lambda$ we denote by $\sigma_\delta$ the bijection from ${\rm LR}_{\lambda,\delta}^{\nu}$ to ${\rm LR}_{\lambda',\delta'}^{\nu'}$; note that this map uses $\delta$ in a crucial way, in Step 3 of its construction. 

\begin{lemma}\label{lemdelta} Consider $b\in{\rm LR}_{\lambda,\delta}^\nu$ with $\delta\in{\mathcal P}_{2n}^{(1,1)}$. The SSYT $\sigma_\delta(b)$ satisfies condition {\rm \eqref{condbranching}} with respect to $\delta$ if and only if $S(b)$ satisfies condition {\rm (C3)}. So in fact, the first condition is  independent of $\delta$. 
\end{lemma}

\begin{proof} Let us denote the first column of $S(b)$ by $(c_1<\ldots<c_p)$, where $p\le n$. By the construction of the map $\sigma_\delta$ in Section~\ref{lrcs}, condition {\rm \eqref{condbranching}} for $\sigma_\delta(b)$ simply means
\[\delta_{c_1}^{\rm rev}\ge\delta_1^{\rm rev}=\delta_2^{\rm rev}\,,\;\ldots\,,\;\delta_{c_p}^{\rm rev}\ge\delta_{2p-1}^{\rm rev}=\delta_{2p}^{\rm rev}\,.\]
We need to show that this is equivalent to
\[c_1\ge 1\,,\;\ldots\,,\;c_p\ge 2p-1\,.\]
The implication $(\Leftarrow)$ is clear since $\delta^{\rm rev}=(\delta_1^{\rm rev}=\delta_2^{\rm rev}\le\delta_3^{\rm rev}=\delta_4^{\rm rev}\le\ldots)$, while $(\Rightarrow)$ is only clear if the weak inequalities defining $\delta$ are strict. 

Assuming that $(\Rightarrow)$ fails, pick the largest $i$ such that $\delta_{c_i}^{\rm rev}=\delta_{2i-1}^{\rm rev}$ and $c_i<2i-1$, where clearly $i\ge 2$; we call such an index $i$ bad. Let us assume first that $c_i=2i-2$, so $\delta_{2i-2}^{\rm rev}=\delta_{2i-1}^{\rm rev}$. Since $b\in{\rm LR}_{\lambda,\delta}^\nu$, we have $\boldsymbol{\varepsilon}(b)\le\delta$, so by \eqref{epsmubstar} we deduce $\varphi_{2i-2}(S(b))=0$. This rules out $i=p$, as well as $i<p$ and $c_{i+1}\ge 2i$, because in these cases $\widetilde{f}_{2i-2}(S(b))\ne 0$, by the usual bracketing rule for crystal operators, see e.g. \cite{HK}. It follows that $c_{i+1}=2i-1$, but this contradicts $c_{i+1}\ge 2(i+1)-1$, which holds by the maximality of $i$. Thus, we must have $c_i\le 2i-3$.

Assuming $i>2$, the index $i-1$ must also be bad, because otherwise we would have 
\[2(i-1)-1\le c_{i-1}<c_i\le 2i-3\,.\]
By repeating the above argument with $i$ replaced by $i-1$, we deduce $c_{i-1}\le 2i-5$. We repeat the previous reasoning for the indices $i-2,i-3,\ldots,2$, and conclude $c_2\le 1$. This leads to the contradiction $1\le c_1<c_2\le 1$, which concludes the proof. 
\end{proof}

\begin{proof}[Proof of Theorem \rm{\ref{thmsums}}] We define the following subset of $S_{b,n}$:
\[\overline{S}_{b,n}:=\{\mu\in S_{b,n}\,:\,\mbox{$\sigma_\delta(b)$ satisfies \eqref{condbranching} with respect to $\delta$}\}\,,\;\;\;\;\mbox{where $\delta:=\boldsymbol{\varepsilon}(b)+\mu$}\,.\] 
Letting 
\[\overline{D}_{2n}(\lambda):=\{b\in D_{2n}(\lambda)\,:\,\mbox{$S(b)$ satisfies condition {\rm (C3)}}\}\,,\]
we observe that its image under $S$ is precisely $\overline{D}_{2n}^*(\lambda)$. By \eqref{sb} and Lemma~\ref{lemdelta}, we have
\begin{equation}\label{sb1}
\overline{S}_{b,n}:=\casetwoother{\mu_{b,n}+{\mathcal P}_{2n}^\boxplus}{b\in\overline{D}_{2n}(\lambda)}{\emptyset}
\end{equation}

We now follow the approach in the proof of {Proposition~\ref{prop6.6}}. This gives
\begin{align*}
\sum_{\nu\in\mathcal{P}_{2n}^{(2)}}\sum_{\delta\in\mathcal{P}_{2n}^{(1,1)}}t^{\left\vert
\nu\right\vert /2}\,\overline{c}_{\lambda,\delta}^{\nu}&=\sum_{b\in D_{2n}(\lambda)}\sum_{\mu\in
\overline{S}_{b,n}}t^{\left\vert \boldsymbol{\varphi}(b)+\mu\right\vert /2}\\
&=\sum_{b\in \overline{D}_{2n}(\lambda)}t^{\left\vert
\boldsymbol{\varphi}(b)+\mu_{b,n}\right\vert /2}\sum_{\kappa\in\mathcal{P}_{2n}^{\boxplus}}t^{\left\vert \kappa\right\vert /2}\\
&=\sum_{b\in \overline{D}_{2n}^*(\lambda)}t^{\left\vert
\boldsymbol{\varepsilon}^*(b)+\mu_{b,n}^*\right\vert /2}\sum_{\kappa\in\mathcal{P}_{2n}^{\boxplus}}t^{\left\vert \kappa\right\vert /2}\,.
\end{align*}
Here the second equality follows from \eqref{sb1}, while the third one follows by translating all the parameters from $\overline{D}_{2n}(\lambda)$ to $\overline{D}_{2n}^*(\lambda)$ via \eqref{epsmubstar}.
\end{proof}
 
We now derive the analogue of {Theorem~\ref{thm6.7}}, and also of
{Theorem~\ref{charge-a}} in type $A$. Observe first we can write more
explicitely for any vertex $b\in B_{2n}(\lambda)$
\[
\left\vert \boldsymbol{\varepsilon}^{\ast}(b)+\mu_{b,n}^{\ast}\right\vert
/2=\sum_{i=1}^{2n-1}(2n-i)\left\lceil \frac{\varepsilon_{i}(b)}{2}\right\rceil
.
\]

\begin{theorem}
\label{kcn} We have
\[
K_{\lambda,0}^{C_{n}}(t)=\sum_{b\in\overline{D}_{2n}^{\ast}(\lambda
)}t^{\mathrm{ch}_{C_{n}}(b)}\,,
\]
where
\[
\mathrm{ch}_{C_{n}}(b)=\sum_{i=1}^{2n-1}(2n-i)\left\lceil \frac{\varepsilon
_{i}(b)}{2}\right\rceil \,.
\]

\end{theorem}

\begin{proof}
The proof is immediate based on Corollary~\ref{kwonbrrule} and {Proposition~\ref{br-rules}~(3)}. Indeed, it suffices to observe that
\[
\sum_{\kappa\in\mathcal{P}_{2n}^{\boxplus}}t^{\left\vert \kappa\right\vert
/2}=\frac{1}{\prod_{i=1}^{n}(1-t^{2i})}\,.
\]
\end{proof}

\subsection{From distinguished tableaux to King tableaux} We follow a similar approach to that in {Section~\ref{dist-tab}}. The goal is to transfer the results to a natural labeling of the vertices of weight $0$ in the type $C_n$ crystal of highest weight $\lambda$, via a bijection with $\overline{D}_{2n}^*(\lambda)$. Such a natural labeling is given by the King tableaux of weight 0 \cite{King}. Recall that the King tableaux of type $C_n$ are just semistandard tableaux of shape $\lambda$ in the alphabet $\{1<\overline{1}<2<\overline{2}<\ldots<n<\overline{n}\}$, with the additional flag condition that the entries in each row $i$ are greater or equal to $i$. The set of such tableaux of weight 0 will be denoted by $K_{C_n}^0(\lambda)$. 

Consider a tableau $b$ in $\overline{D}_{2n}^*(\lambda)$, and let $N_i(b)$ denote the number of entries equal to $i$. Note first that conditions (C1) and (C2) in Section~\ref{ranknlusz} can be phrased as the following more explicit ones, for $i=1,\ldots,n$:
\begin{enumerate}
\item[(C1$'$)] the subword of the Japanese reading of the tableau $b$ formed by $2i-1$ and $2i$ has the property that in each right factor the number of $2i-1$ is less or equal to the number of $2i$; 
\item[(C2$'$)] $N_{2i}(b)-N_{2i-1}(b)$ is a (non-negative) even integer.
\end{enumerate}
Condition (C2) is also equivalent to the fact that the rows of $\boldsymbol{\theta}^*_n(b):=\boldsymbol{\varepsilon}^*(b)+\mu_{b,n}^*$ have even lengths. 

Given $b$ as above, we will map it to a King tableau in $K_{C_n}^0(\lambda)$. Letting $k_i:=N_{2i}(b)-N_{2i-1}(b)$, we apply the crystal operator $\widetilde{e}_{2i-1}^{k_i/2}$ to $b$, for $i=1,\ldots,n$. Note that these operators commute, and in fact they correspond to a $U_q(\mathfrak{sl}_2\oplus\ldots\oplus\mathfrak{sl}_2)$-crystal structure, cf. {Section~\ref{dist-tab}}. Afterwards, we replace the entries $2i-1$ and $2i$ with $i$ and $\overline{\imath}$, respectively, for each $i$. It is easy to see that the resulting filling has weight $0$, and that the flag condition (C3) turns into the similar condition for King tableaux. So the result is in $K_{C_n}^0(\lambda)$. 

Moreover, this map has an inverse. Indeed, given a King tableau $T$, we first replace the entries $i$ and $\overline{\imath}$ with $2i-1$ and $2i$, respectively. Then we map the resulting filling to the lowest weight element with respect to the corresponding $U_q(\mathfrak{sl}_2\oplus\ldots\oplus\mathfrak{sl}_2)$-crystal structure. It is easy to see that the resulting filling is in $\overline{D}_{2n}^*(\lambda)$. For obvious reasons, we denote this map by $T\mapsto L(T)$. 

Based on the above discussion, Theorem~\ref{kcn} can be rephrased as follows.

\begin{theorem}
\label{kfking} We have
\[
K_{\lambda,0}^{C_{n}}(t)=\sum_{T\in K_{C_{n}}^{0}(\lambda)}t^{\mathrm{ch}
_{C_{n}}(L(T))}\,,
\]
where
\[
\mathrm{ch}_{C_{n}}(L(T))=\sum_{i=1}^{2n-1}(2n-i)\left\lceil \frac
{\varepsilon_{i}(L(T))}{2}\right\rceil \,.
\]

\end{theorem}

\begin{remarks}\label{remkn}{\rm 
(1) As noted in {Remark~\ref{rem11.5}}~(1), there does not seem to be a simple way to express the related statistic above directly in terms of $T$. However, the map $T\mapsto L(T)$ is a simple one. 

(2) Theorem~\ref{kfking} shows that it is more natural to define a statistic for computing the Kostka-Foulkes polynomial on King tableaux, rather than on the other important set of symplectic tableaux, namely the Kashiwara-Nakashima (KN) tableaux \cite{HK}. A natural question is whether the statistic above can be translated to the KN tableaux via the bijection in \cite{Sh}, and moreover if one recovers in this way the charge statistic constructed in \cite{lec} (which conjecturally computes the Kostka-Foulkes polynomials); we will be investigating this question in the future. 
}
\end{remarks}

We have the following analogue of {Theorem~\ref{thm6.8}}, cf. also Remark~\ref{rem11.5}, related to the expression of the multivariable generalization of $K_{\lambda,0}^{C_n}(t)$, denoted $K_{\lambda,0}^{C_n}(\boldsymbol{t})$. Like in the infinite case, the related combinatorial expression follows immediately from the (finite type) combinatorics worked out above. Note that the discrepancy mentioned in Assertion 2 of {Remark~\ref{rem11.5}} has now been corrected by passing from the set of distinguished tableaux $D_{2n}(\lambda)$ to its image $D_{2n}^*(\lambda)$ under the Sch\"utzenberger involution. 

\begin{theorem} Define the multivariable polynomial $K_{\lambda,0}^{C_n}(\boldsymbol{t})$ by%
\[
\frac{K_{\lambda,0}^{C_n}(\boldsymbol{t})}{\prod_{i=1}^{n}(1-t_{2i})}%
=\sum_{\nu\in\mathcal{P}^{(2)}_{2n}}\sum_{\delta\in\mathcal{P}^{(1,1)}_{2n}%
}\boldsymbol{t}^{\frac{1}{2}\nu}c_{\lambda,\delta}^{\nu}\,.
\]Then we have
\[
K_{\lambda,0}^{C_n}(\boldsymbol{t})=\sum_{T\in K_{C_n}^0(\lambda)}{\boldsymbol{t}}^{\boldsymbol{\theta}_n^*(L(T)) /2}\,,
\]
where 
\[{\boldsymbol{t}}^{\boldsymbol{\theta}_n^*(L(T)) /2}=\prod_{i=1}^{2n-1} t_{2n-i}^{\lceil{\varepsilon_{i}(L(T))}/{2}\rceil}\,.\]
\end{theorem}

We will now continue Example~\ref{ex-11}.

\begin{example}\cellsize=3.1ex
{\rm 
Assume $\lambda=(1,1)$ in type $C_n$. Then we get
\[
K_{C_{n}}^0(\lambda
)=\left\{\tableau{k\\ \overline{k}}
\mid k=2,\ldots,n\right\}  .
\]
This gives%
\[
L\left(\tableau{k\\ \overline{k}}
\right)  =%
\begin{tabular}
[c]{|c|}\hline
$2k-1$\\\hline
$2k$\\\hline
\end{tabular}
\;\;\text{ and }\;\;\boldsymbol{\varepsilon}^*\left(
\begin{tabular}
[c]{|c|}\hline
$2k-1$\\\hline
$2k$\\\hline
\end{tabular}
\right)  =\omega_{2(n-k+1)}\text{ for any }k=2,\ldots,n.
\]
Therefore
\[
\boldsymbol{\theta}_n^*\left(
\begin{tabular}
[c]{|c|}\hline
$2k-1$\\\hline
$2k$\\\hline
\end{tabular}
\right)  =2\omega_{2(n-k+1)}\text{ for any }k=2,\ldots,n\text{.}%
\]
Finally%
\[
K_{\lambda,0}^{C_{n}}(\boldsymbol{t})=\sum_{k=2}^n t_{2(n-k+1)}=\sum_{k=1}^{n-1} t_{2k}\;\;\text{ and }\;\;K_{\lambda,0}^{C_{n}}(t)=\sum_{k=1}^{n-1}t^{2k}=\frac{t^{2}-t^{2n}}{1-t^{2}}\,.
\]
}
\end{example}

\section{Three applications}

In this section, we present three applications of Theorem~\ref{kfking}. 

\subsection{Growth of generalized exponents}

First we analyze the growth of the generalized exponents of type $C_{n}$ with respect to the rank $n$. 

The (weight $0$) symplectic King tableaux of type $C_{n}$ embed into those of type $C_{n+1}$ by changing the entries $k,\overline{k}$ to $k+1,\overline{k+1}$, for all $k$, respectively. Moreover, it is easy to see that this map preserves the statistic in Theorem~\ref{kfking}. So we obtain the following result, which to our knowledge is new. 

\begin{theorem}
For any integer $n$ and any partition $\lambda$ with at most $n$ parts, we have
$K_{\lambda,0}^{C_{n+1}}(t)-K_{\lambda,0}^{C_{n}}(t)\in\mathbb{Z}_{\geq0}[t]$.
\end{theorem}

\subsection{Reducing a type $C$ generalized exponent to one of type $A$} We now prove a conjecture of the first author \cite{lec}. This conjecture is the first step in the construction of the type $C_{n}$ charge statistic in \cite{lec}, and proves the conjecture that this charge computes the corresponding Kostka-Foulkes polynomials in the case of column shapes; see Remark~\ref{remkn}~(2). 

We now label the Dynkin diagram of type $C_n$ such that the special node is $n$. Consider the fundamental weight $\omega_{2p}$, where $p\in\{1,\ldots,\lfloor n/2\rfloor\}$. All the zero weight vertices in the crystal $B(\omega_{2p})$ belong to the same type $A_{n-1}$ component, which has highest weight $\gamma_p:=\varepsilon_1+\ldots+\varepsilon_p-\varepsilon_{n-p+1}-\ldots-\varepsilon_n$, where $\varepsilon_i$ are the coordinate vectors in ${\mathbb R}^n$. In type $A_{n-1}$, this weight corresponds to the partition $(1^{n-2p},2^p)$. 

\begin{theorem}\label{ac} We have
\[K^{C_n}_{\omega_{2p},0}(t)=K_{\gamma_p,0}^{A_{n-1}}(t^2)\,.\]
\end{theorem}

Before proving this theorem, we need to describe the KN tableaux for some column shape $(1^k)$ \cite{HK}, which index the vertices of the type $C_n$ crystal $B(\omega_k)$. 

\begin{definition} {\rm A column-strict filling $C=(c_1<\ldots <c_k)$ with entries in $\{1<\ldots<n<\overline{n}<\ldots<\overline{1}\}$ is a KN column if there
is no pair $(z,\overline{z})$ of letters in $C$ such that: 
\[z = c_p\,,\;\;\;\;\;\overline{z} = c_q\,,\;\;\;\;\;q-p\le k - z\,.\]}
\end{definition}

We will need a different definition of KN columns, which was proved to be equivalent to the one above in \cite{Sh}.

\begin{definition}\label{defkn}{\rm 
 Let $C$ be a column and $I=\{x_1 > \ldots > x_r\}$ the set of unbarred letters $z$ such that
the pair $(z,\overline{z})$ occurs in $C$. The column $C$ can be split when there exists a set
of $r$ unbarred letters $J = \{y_1 > \ldots > y_r\} \subset\{1,\ldots,n\}$ such that:
\begin{itemize}
\item $y_1$ is the greatest letter in $\{1,\ldots,n\}$ satisfying: $y_1 < x_1$, $y_1\not\in C$, and $\overline{y_1}\not\in C$,
\item for $i = 2, ..., r$, the letter $y_i$ is the greatest one in $\{1,\ldots,n\}$ satisfying $y_i < \min(y_{i-1},x_i)$, $y_i\not\in C$, and $\overline{y_i}\not\in C$.
\end{itemize}
In this case, we say that $x_i$ is paired with $y_i$, and we write:
\begin{itemize}
\item $lC$ for the column obtained by changing $x_i$ into $y_i$ in $C$ for each letter $x_i\in I$, and by reordering if
necessary;
\item $rC$ for the column obtained by changing  $\overline{x_i}$ into $\overline{y_i}$ in $C$ for each letter $x_i\in I$, and by reordering if
necessary.
\end{itemize}
The pair $(lC,rC)$ will be called a split column.}
\end{definition}

\cellsize=3.1ex

\begin{example}\label{exkn}{\rm
The following is a KN column of height $5$ in type $C_n$ for $n\ge 5$, together with the corresponding split column:
\[C=\tableau{{4}\\{5}\\{\overline{5}}\\{\overline{4}}\\{\overline{3}}}\,,\;\;\;\;\;(lC,rC)=\tableau{{1}&{4}\\{2}&{5}\\{\overline{5}}&{\overline{3}}\\{\overline{4}}&{\overline{2}}\\{\overline{3}}&{\overline{1}}}\,.\]
We used the fact that $I=\{5>4\}$, so $J=\{2>1\}$. 
}
\end{example}

\cellsize=2.5ex

For the definition of the crystal operators on KN columns via the well-known bracketing rule, we refer to \cite{HK}.

\begin{proof}[Proof of Theorem {\rm \ref{ac}}] We use the King tableaux for computing $K^{C_n}_{\omega_{2p},0}(t)$ via Theorem~\ref{kfking}. Meanwhile, $K_{\gamma_p,0}^{A_{n-1}}(t)$ is computed based on an analogue of Theorem~\ref{charge-a}, namely
\begin{equation}\label{ccha}
K_{\lambda,0}^{A_{n-1}}(t)=\sum_{b\in B(\lambda)_{0}}t^{\sum_{i=1}^{n-1}(n-i)\varepsilon_i(b) }\,,
\end{equation}
which is referred to \cite{LLT}. For this computation, we use the crystal structure on the type $A_{n-1}$ component of highest weight $\gamma_p$ of $B(\omega_{2p})$, which contains the zero weight KN tableaux. 

First we need a bijection between the zero weight King tableaux and KN tableaux of shape $(1^{2p})$. Let $C_K=(c_1<\overline{c_1}<\ldots< c_p<\overline{c_p})$ be such a King tableau, which means that $c_i\ge 2i-1$ and $\overline{c_i}\ge 2i$, for $i=1,\ldots,p$; but these conditions are equivalent to $c_i\ge 2i$. Let $C_{KN}=(d_1<\ldots< d_p<\overline{d_p}<\ldots<\overline{d_1})$ be a zero weight KN column, where we note the different order used on the alphabet $\{1,\ldots,n,\overline{n},\ldots,\overline{1}\}$. The condition in Definition~\ref{defkn} implies that $d_i\ge 2i$ for any $i$, because $d_1,\ldots d_i$ need to be paired with distinct entries strictly less than $d_i$, which are also different from $d_1,\ldots,d_{i-1}$. One can check that the reciprocal is also true. Thus the desired bijection maps $C_K$ to $C_{KN}$ with $d_i=c_i$, which we now assume. 

Now let us calculate the exponent of the variable $t$ corresponding to $C_K$ in $K^{C_n}_{\omega_{2p},0}(t)$, as given by Theorem~\ref{kfking}. First we replace $c_i$ by $2c_i-1$ and $\overline{c_i}$ by $2c_i$, obtaining a column $C_K'$. Note that this is both a highest and lowest weight element with respect to the corresponding $U_q(\mathfrak{sl}_2\oplus\ldots\oplus\mathfrak{sl}_2)$-crystal structure, so $L(C_K)=C_K'$. Let $P:=\{c_i\in C_K\,|\,c_i-1\not\in C_K\}$. Note that the only type $A$ raising crystal operators which can be applied to $C_K'$ are $\widetilde{e}_{2p-2}$ for $p\in P$, and each can be applied only once. Thus, for each $p\in P$, we get a contribution of $2(n-p+1)$ to the mentioned exponent of $t$.

Finally, let us calculate the exponent of $t$ corresponding to $C_{KN}$ in $K_{\gamma_p,0}^{A_{n-1}}(t)$, as mentioned above, based on \eqref{ccha}. Let $C_{KN}^+:=(c_1<\ldots< c_p)$. Observe first that 
\[\varepsilon_{p-1}(C_{KN})=\varepsilon_{p-1}(C_{KN}^+)=\casetwoother{1}{p\in P}{0}\]
This means that, for each $p\in P$, we get a contribution of $n-p+1$ to the mentioned exponent of $t$. This concludes the proof. 
\end{proof}

\begin{remark}{\rm 
Theorem \ref{ac} also permits to establish the conjecture of \cite{lec} for
Lusztig $t$-analogues associated to any fundamental weight. Indeed, each
such fundamental weight is indexed by a column partition $\lambda
=(1^{k})=\omega _{k}$ with $k\leq n$ and the possible corresponding dominant
weights yielding nonzero polynomials have the form $\mu =(1^{a})=\omega _{a}$
where $k-a$ is a nonnegative even integer. We then have 
\begin{equation*}
K_{\omega _{k},\omega _{a}}^{C_{n}}(t)=K_{\omega _{k-a},0}^{C_{n-a}}(t).
\end{equation*}%
This follows in fact from a more general row removal property of Lusztig $t$%
-analogues of type $C_{n}$. Assume that $\lambda $ and $\mu $ are two
partitions such that $\lambda _{1}=\mu _{1}$ then%
\begin{equation*}
K_{\lambda ,\mu }^{C_{n}}(t)=K_{\lambda ^{\flat },\mu ^{\flat }}^{C_{n-1}}(t)
\end{equation*}%
where $\lambda ^{\flat }$ and $\mu ^{\flat }$ are the partitions obtained by
removing the part $\lambda _{1}=\mu _{1}$ in $\lambda $ and $\mu $,
respectively. This can be proved directly from the very definition of $%
K_{\lambda ,\mu }^{C_{n}}(t)$ in terms of partition function or by using the
Morris type recurrence formula established in \cite{lec}.}
\end{remark}

\subsection{The smallest power of $t$ in $K^{C_n}_{\lambda,0}(t)$}\label{S:smallestpower}

The largest power of $t$ in $K^{C_n}_{\lambda,0}(t)$ is well-known to be $\langle\lambda,\rho^\vee\rangle$, where  $\rho^\vee$ is half the sum of the positive coroots. Furthermore, it is also known that the smallest power is greater or equal to $|\lambda|/2$. See \cite{Lec2,Lec4}. As the third application of our formula for $K^{C_n}_{\lambda,0}(t)$, we will determine this smallest power. 

Let $\lambda\in{\mathcal P}_n$ be such that $|\lambda|$ is even, and write $\lambda=\sum_{i=1}^n a_i\,\omega_{n+1-i}$. Define
\[s_k:=\sum_{i=1}^k a_i\,,\;\;\;\;\;b_i:=\casethree{a_i+1}{\mbox{$a_i$ odd and $s_i$ odd}}{a_i-1}{\mbox{$a_i$ odd and $s_i$ even}}{a_i}{\mbox{$a_i$ even}}\]
Also let $s_0:=0$ and $S:=s_n$.

\begin{theorem}\label{minpower} The smallest power of $t$ in $K^{C_n}_{\lambda,0}(t)$ is
\begin{equation}\label{expminpower}\frac{1}{2}\sum_{i=1}^n (n+1-i) b_i=\frac{|\lambda|}{2}+\frac{1}{2}\sum_{i\,:\,a_i\:\rm{odd}}(-1)^{s_i-1}(n+1-i)\,.\end{equation}
\end{theorem}

We start by sketching the idea of the proof, whose details can be found in the next section. Based on Theorem~\ref{kcn}, we need to find the filling $\sigma\in\overline{D}_{2n}^{\ast}(\lambda)$ which minimizes 
\[\mathrm{ch}_{C_{n}}(\sigma)=\sum_{i=1}^{2n-1}(2n-i)\left\lceil \frac{\varepsilon _{i}(\sigma)}{2}\right\rceil \,.\]

We will first minimize $\mathrm{ch}_{C_{n}}(\sigma)$ for fillings $\sigma$ of the row shape $(S)$ with $1,\ldots,2n$, subject to certain conditions. Namely, let $\Sigma$ be the set of all $\sigma=(\sigma_1\le\ldots\le\sigma_S)$ satisfying
\begin{equation}\label{defsigma}\sigma_i\le n+k\,,\;\;\;\;\;\mbox{for $s_{k-1}<i\le s_k$, and $k=1,\ldots,n$}\,.\end{equation}
Note that this condition is a necessary one for the first row of a filling of $\lambda$ with $1,\ldots,2n$. Let us also define the sequence $c_1,\ldots,c_n$ by setting $c_i:=b_i$, except for the case in which, for the largest $i$ with $a_i$ odd, we have $s_i$ odd, in which case $c_i:=a_i$ (and $b_i:=a_i+1$). Note that $a_1+\ldots+a_n=c_1+\ldots+c_n=S$. 

\begin{lemma}\label{minrow} We have
\[\min_{\sigma\in\Sigma}\,\mathrm{ch}_{C_{n}}(\sigma)=\frac{1}{2}\sum_{i=1}^n (n+1-i) b_i\,,\]
and the minimum is attained for $\sigma_{\min}^{\rm{row}}:=((n+1)^{c_1}(n+2)^{c_2}\ldots(2n)^{c_n})$. 
\end{lemma}

Now consider $\sigma\in\overline{D}_{2n}^{\ast}(\lambda)$. By the usual bracketing rule for crystal operators, see e.g. \cite{HK}, it is easy to see that all entries $i\ge 2$ in the first row of $\sigma$ contribute to $\varepsilon_{i-1}(\sigma)$. Thus, it suffices to construct $\sigma_{\min}\in\overline{D}_{2n}^{\ast}(\lambda)$ whose first row is $\sigma_{\min}^{\rm{row}}$, and for which no entry $i$ below the first row contributes to  $\varepsilon_{i-1}(\sigma_{\min})$. This is achieved with one mild failure of the last property; nevertheless, we always have $\mathrm{ch}_{C_{n}}(\sigma_{\min}^{\rm{row}})=\mathrm{ch}_{C_{n}}(\sigma_{\min})$, which is all that is needed.

Algorithm~\ref{algmin} describes the construction of $\sigma_{\min}$. In order to state it, we need some definitions and related results. Let $k_1<k_2<\ldots<k_p$ be the indexes $i$ for which $a_i$ is odd. We pair them from left to right  as $(k_1,k_2)$, $(k_3,k_4)$, $\ldots$, where $k_p$ is unpaired if $p$ is odd. Given such a pair $(k,k')$, we say that all the columns in $\lambda$ of heights $n+1-i$ with $k\le i\le k'$ form a block. This block is called odd or even, depending $k'-k$ being odd or even, respectively. A subblock of columns is formed by all columns of the same height in a given block. If $p$ is odd, we say that all columns of height at most $n+1-k_p$ form an incomplete block. 

We call a column of $\lambda$ special if it is the first one in a subblock, without being the first one of the corresponding block. Note that, if the first row of $\lambda$ is filled with the entries of $\sigma_{\min}^{\rm{row}}$, then a column is special if and only if its top entry is strictly smaller than the maximum possible, namely $n+i$ if $n+1-i$ is the corresponding column height. We call a special column odd if its top entry has the same parity as the column height. Note that this condition on a special column is equivalent to the bottom entry being odd (hence the name), when the column is filled with consecutive entries starting from the top one. 

\begin{lemma}\label{oddeven} {\rm (1)} The number of odd blocks is even unless $p$ and $n+1-k_p$ are odd.

{\rm (2)} The number of odd special columns in a block is odd or even, depending on the block being odd or even, respectively.

{\rm (3)} The total number of odd special columns is even unless $p$ and $n+1-k_p$ are odd.
\end{lemma}

\begin{algorithm}\label{algmin} {\rm Construction of $\sigma_{\min}$.
\begin{description}
\item[Step 1] Fill the first row of $\lambda$ with the entries of $\sigma_{\min}^{\rm{row}}$. 
\item[Step 2] Fill all columns except the odd special ones with consecutive entries starting from the top entry. 
\item[Step 3] Fill the odd special columns, considered from right to left, as follows. 
\begin{itemize}
\item If the last entry of the previously filled odd special column (assuming it exists) is $2i-1$, then the current one will contain $2i$, but not $2i-1$. 
\item The above rule is also applied to the rightmost odd special column if $p$ and $n+1-k_p$ are odd, where $2i=n+k_p$ is the top entry in each column of height $n+1-k_p$.
\item With the above rules in place, fill the current special column by considering consecutive entries starting from the top one.
\end{itemize}
\end{description} }
\end{algorithm}

\begin{lemma}\label{minfill} The filling $\sigma_{\min}$ belongs to $\overline{D}_{2n}^{\ast}(\lambda)$. Furthermore, no entry $i$ below the first row contributes to  $\varepsilon_{i-1}(\sigma_{\min})$ with one exception: if $p$ and $n+1-k_p$ are odd, then one entry $n+k_p$ below the first row contributes to $\varepsilon_{n+k_p-1}(\sigma_{\min})$. In addition, we always have  $\mathrm{ch}_{C_{n}}(\sigma_{\min}^{\rm{row}})=\mathrm{ch}_{C_{n}}(\sigma_{\min})$.
\end{lemma}

\begin{proof}[Proof of Theorem~\rm{\ref{minpower}}] By Lemma~\ref{minfill}, $\mathrm{ch}_{C_{n}}(\sigma_{\min})$ is given by the expression in \eqref{expminpower}. Lemma~\ref{minrow} guarantees that this is the minimum of the charge over $\overline{D}_{2n}^{\ast}(\lambda)$. Thus, Theorem~\ref{minpower} is proved. 
\end{proof}

In conclusion, all that is left is to prove Lemmas~{\rm \ref{minrow}},~\ref{oddeven},~and~{\rm \ref{minfill}}, which is done in Section~\ref{prooflemmas}.

We will now give an example of the construction of $\sigma_{\min}$. We will also exhibit a second filling, with the same shape and first row as $\sigma_{\min}$, which also satisfies the properties in Lemma~\ref{minfill}. This will have the same charge as $\sigma_{\min}$, which shows that the coefficient of the smallest power of $t$ in $K^{C_n}_{\lambda,0}(t)$ can be strictly larger than $1$.

\cellsize=3.5ex

\begin{example}{\rm Let $n=5$ and $\lambda=(7,6,5,3,1)$. The sequence $(a_i)$ is $(1,2,2,1,1)$, and thus there is a single block, which is odd and consists of all the columns of $\lambda$ except the last one. The special columns are the second, the fourth, and the sixth; they are all odd. The sequences $(b_i)$ and $(c_i)$ are $(2,2,2,0,2)$ and $(2,2,2,0,1)$, respectively, while $\sigma_{\min}^{\rm{row}}=(6,6,7,7,8,8,10)$. The filling $\sigma_{\min}$ and a different one with the same charge are 
\[  \tableau{6&6&7&{{7}}&8&{{8}}&{10}\\7&7&{{8}}&8&9&10\\8&8&9&{\mathbf{9}}&10\\9&{\mathbf{10}}&10\\10} \,,\;\;\;\;\;  \tableau{6&6&7&{{7}}&8&{{8}}&{10}\\7&7&{{8}}&8&9&{\mathbf{9}}\\8&8&9&{\mathbf{10}}&10\\9&{{10}}&10\\10}  \,.\]
It is straightforward to check that the above fillings satisfy the properties in Lemma~\ref{minfill}; in particular, the highlighted entries are bracketed in the usual procedure for applying crystal operators. Thus, both of these fillings have charge $|\lambda|/2+1/2(5-2+1)=11+2=13$. 

Below is a different filling $\sigma_{\min}$, which illustrates another aspect of Algorithm~\ref{algmin}.
\[\tableau{6&6&9&{\mathbf{9}}\\7&{\mathbf{10}}&10\\8\\9\\10}\,.\]
}\end{example}

\subsection{The proof of the lemmas in Section~{\rm \ref{S:smallestpower}}}\label{prooflemmas} 

The terminology and notation in the previous section will be used. We start with Lemma \ref{minrow}. We first define the following moves $\sigma\rightarrow\sigma'$ on sequences $\sigma=(i^{m_i})_{i=1,\ldots,2n}$ in $\Sigma$, assuming that $\sigma'$ is still in $\Sigma$:
\begin{enumerate}
\item $(\ldots,i^{k+2},\ldots)\rightarrow(\ldots,i^{k},i+1,i+1,\ldots)$;
\item $(\ldots,i^{2k+1},\ldots)\rightarrow(\ldots,i^{2k},i+1,\ldots)$;
\item $(\ldots,i^{2k-1},j^{l+1},\ldots)\rightarrow(\ldots,i^{2k},j^l,\ldots)$;
\item $(\ldots,i^{2k},j^{2l-1},\ldots)\rightarrow(\ldots,i^{2k-1},j^{2l},\ldots)$.
\end{enumerate}
It is not hard to see that in all cases we have 
\begin{equation}\label{chsmaller}
\mathrm{ch}_{C_{n}}(\sigma)\ge \mathrm{ch}_{C_{n}}(\sigma')\,;
\end{equation}
moreover, in case (4) we always have equality. Indeed, note that
\[\mathrm{ch}_{C_{n}}(\sigma)=\sum_{i=2}^{2n}(2n+1-i)\left\lceil\frac{m_i}{2}\right\rceil\,;\]
based on this, it suffices to observe that if we insert an entry $i>1$ into a sequence $\sigma$, the charge increases by $2n+1-i$, if $m_i$ is even, and does not change, otherwise. 

It is helpful to visualize the moves (1)-(4) using the following representation of a sequence $\sigma=(i^{m_i})_{i=1,\ldots,2n}$ in $\Sigma$ as a lattice path from $(0,0)$ to $(S,2n)$ with steps $(1,0)$ and $(0,1)$. The horizontal segments in this path are 
\[(m_1+\ldots+m_{i-1},i)\rightarrow(m_1+\ldots+m_i,i)\,,\;\;\;\;\mbox{for $m_i>0$, $i=1,\ldots,2n\,.$}\]
Note that condition~\eqref{defsigma} defining $\Sigma$ simply means that this path stays weakly below the similar path from $(0,0)$ to $(S,2n)$, whose horizontal segments are 
\[(s_{i-1},n+i)\rightarrow(s_i,n+i)\,,\;\;\;\;\mbox{for $a_i>0$, $i=1,\ldots,n\,.$}\]
The latter path will be called the upper bound path. We will also consider the path corresponding to $\sigma_{\min}^{\rm{row}}$, which will be called the target path. 

Now recall that $k_1<k_2<\ldots<k_p$ are the indexes $i$ for which $a_i$ is odd, which are paired $(k_1,k_2)$, $(k_3,k_4)$, etc. For each such pair $(k,k')$, we consider the subpath of the upper bound path between the horizontal segments with $y=n+k$ and $y=n+k'$, inclusive, plus the vertical segment after the last horizontal one; we call it an odd subpath. If $p$ is odd, the subpath between the horizontal segment with $y=n+k_p$ and the end of the path is called an incomplete odd subpath. The subpaths obtained by removing the odd ones are called even. 

Now note that the upper bound path and the target one are closely related. Namely, every even subpath of the former coincides with a corresponding subpath of the latter, and so does the incomplete odd subpath (if any). Moreover, for every odd subpath of the former, there is a corresponding one in the latter  whose vertical segments are translations by $(1,0)$ of the vertical segments of the former; the exception are the last vertical segments in the two paths, which coincide. Thus, we can also divide the target path into even, odd, and incomplete odd subpaths.

\begin{proof}[Proof of Lemma~\rm{\ref{minrow}}] In terms of the above visualization, and based on~\eqref{chsmaller}, it suffices to show that any path that is weakly below the upper bound path (including the latter) can be related to the target path by applying the moves (1)-(4). This can be done as follows, using a sequence of intermediate paths. See Example~\ref{expaths} for an illustration of this procedure.  
\begin{description}
\item[Step 1] By applying the moves (1) and (2), from southwest to northeast in the current path, we obtain a path in which every vertical segment coincides with the corresponding one of the upper bound path, or with its translation by $(1,0)$; moreover, for the first vertical segment (starting at $(0,0)$), the first statement holds. Divide the obtained path into even, odd, and incomplete odd subpaths.
\item[Steps 2-4] These steps are applied to the subpaths of the path in Step 1, considered from southwest to northeast. As a result, each subpath will coincide with the corresponding one of the target path.
\item[Step 2] The moves (2) are applied to an even subpath. 
\item[Step 3] The moves (2) and (3) are applied to an odd subpath.
\item[Step 4] The moves (2) and (4) are applied to the incomplete odd subpath (if any).
\end{description}
\end{proof}

\begin{example}\label{expaths}{\rm Let $n=7$, and let the sequence $(a_i)$ be $(1,0,2,1,1,2,2)$. We have $\sigma_{\min}^{\rm{row}}=(8,8,10,10,12,13,13,14,14)$, which corresponds to the target path, while the upper bound path corresponds to $(8,10,10,11,12,13,13,14,14)$. Both of these paths consist of an odd subpath and an incomplete odd subpath. Consider $\sigma=(7,7,7,9,10,11,11,12,14)$. Its corresponding path is represented in the first diagram below, whose bottom left corner has coordinates $(0,7)$, while the upper bound path appears in all three diagrams. The second diagram represents the result of Step 1 in the above algorithm; move (1) was applied six times, while move (2) twice. The last diagram contains the target path, which is obtained from the path in the second diagram via Step 3 followed by Step 4. In Step 3, move (3) was applied twice, and after that move (2) once; in Step 4, move (4) was applied twice (from northeast to southwest). 

\setlength{\unitlength}{0.5cm}
$\!\!$\begin{picture}(11,9)
\linethickness{0.075mm}
\multiput(0,0)(1,0){10}{\line(0,1){7}}
\multiput(0,0)(0,1){8}{\line(1,0){9}}
\linethickness{0.4mm}
\put(0,0){\line(0,1){1}}\put(0,1){\line(1,0){1}}
\put(1,1){\line(0,1){2}}\put(1,3){\line(1,0){2}}
\put(3,3){\line(0,1){1}}\put(3,4){\line(1,0){1}}
\put(4,4){\line(0,1){1}}\put(4,5){\line(1,0){1}}
\put(5,5){\line(0,1){1}}\put(5,6){\line(1,0){2}}
\put(7,6){\line(0,1){1}}\put(7,7){\line(1,0){2}}
\put(0,0){\line(1,0){3}}\put(3,0){\line(0,1){2}}
\put(3,2){\line(1,0){1}}\put(4,2){\line(0,1){1}}
\put(4,3){\line(1,0){1}}\put(5,3){\line(0,1){1}}
\put(5,4){\line(1,0){2}}\put(7,4){\line(0,1){1}}
\put(7,5){\line(1,0){1}}\put(8,5){\line(0,1){2}}
\put(8,7){\line(1,0){1}}
\end{picture}
$\!$\begin{picture}(11,9)
\linethickness{0.075mm}
\multiput(0,0)(1,0){10}{\line(0,1){7}}
\multiput(0,0)(0,1){8}{\line(1,0){9}}
\linethickness{0.4mm}
\put(0,0){\line(0,1){1}}\put(0,1){\line(1,0){1}}
\put(1,1){\line(0,1){2}}\put(1,3){\line(1,0){2}}
\put(3,3){\line(0,1){1}}\put(3,4){\line(1,0){1}}
\put(4,4){\line(0,1){1}}\put(4,5){\line(1,0){1}}
\put(5,5){\line(0,1){1}}\put(5,6){\line(1,0){2}}
\put(7,6){\line(0,1){1}}\put(7,7){\line(1,0){2}}
\put(0,0){\line(0,1){1}}\put(0,1){\line(1,0){1}}
\put(1,1){\line(0,1){2}}\put(1,3){\line(1,0){2}}
\put(3,3){\line(0,1){1}}\put(3,4){\line(1,0){2}}
\put(5,4){\line(0,1){1}}\put(5,5){\line(1,0){1}}
\put(6,5){\line(0,1){1}}\put(6,6){\line(1,0){2}}
\put(8,6){\line(0,1){1}}\put(8,7){\line(1,0){1}}
\end{picture}
$\!$\begin{picture}(11,9)
\linethickness{0.075mm}
\multiput(0,0)(1,0){10}{\line(0,1){7}}
\multiput(0,0)(0,1){8}{\line(1,0){9}}
\linethickness{0.4mm}
\put(0,0){\line(0,1){1}}\put(0,1){\line(1,0){1}}
\put(1,1){\line(0,1){2}}\put(1,3){\line(1,0){2}}
\put(3,3){\line(0,1){1}}\put(3,4){\line(1,0){1}}
\put(4,4){\line(0,1){1}}\put(4,5){\line(1,0){1}}
\put(5,5){\line(0,1){1}}\put(5,6){\line(1,0){2}}
\put(7,6){\line(0,1){1}}\put(7,7){\line(1,0){2}}
\put(0,0){\line(0,1){1}}\put(0,1){\line(1,0){2}}
\put(2,1){\line(0,1){2}}\put(2,3){\line(1,0){2}}
\put(4,3){\line(0,1){2}}\put(4,5){\line(1,0){1}}
\put(5,5){\line(0,1){1}}\put(5,6){\line(1,0){2}}
\put(7,6){\line(0,1){1}}\put(7,7){\line(1,0){2}}
\end{picture}

}
\end{example}

We conclude by proving Lemmas \ref{oddeven} and \ref{minfill}.

\begin{proof}[Proof of Lemma~\rm{\ref{oddeven}}] It is not hard to see that the number of boxes in an odd (resp. even) block is odd (resp. even); in addition, if $p$ is odd, the number of boxes in the incomplete block is even unless $n+1-k_p$ is odd (recall that this number represents the height of the first column in the incomplete block). Based on this and the fact that $|\lambda|$ is even, the first statement is immediate. 

Now let us consider a block corresponding to a pair $(k,k')$, i.e., it contains all columns of heights $n+1-i$ for $k\le i\le k'$. A special column is odd or even depending on the difference between its height and the height of the previous column being odd or even. But the sum of all these numbers is the difference between the heights of the first and last columns in the block, namely $(n+1-k)-(n+1-k')=k'-k$. The second statement now follows from the fact that the parity of a block is determined by $k'-k$. The third statement is an immediate consequence of the first two.
\end{proof}

\begin{proof}[Proof of Lemma~\rm{\ref{minfill}}] It is not hard to see that the filling $\sigma_{\min}$ is a semistandard Young tableau satisfying the flag condition (C3) in Section~\ref{ranknlusz}. Indeed, for semistandardness, observe first that if the first row of $\lambda$ is $\sigma_{\min}^{\rm{row}}$ and we fill all columns with consecutive entries starting from the top one, we clearly obtain a semistandard tableau. To obtain $\sigma_{\min}$, the only change we need is a certain increase in the entries of every other odd special column starting with the leftmost one. But in each case the entries of the next column are the largest possible, so the weakly increasing condition for rows is still verified. 

To complete the proof, it suffices to check the following properties for the column word of  $\sigma_{\min}$, the first two of which rely on the usual bracketing rule for crystal operators~\cite{HK}.
\begin{enumerate}
\item[(P1)] After bracketing $(2i,2i-1)$, there is no unbracketed $2i-1$; also, there is no unbracketed $2i$ below the first row with one exception: a single $2i=n+k_p$ if $p$ and $n+1-k_p$ are odd.
\item[(P2)] For any pair $(2i+1,2i)$, there is no unbracketed $2i+1$ below the first row.
\item[(P3)] Each even entry in the first row occurs an even number of times with one exception: $n+k_p$ if $p$ and $n+1-k_p$ are odd. 
\end{enumerate}

Property (P3) is immediate from the construction of $\sigma_{\min}^{\rm{row}}$. By analyzing Algorithm~\ref{algmin}, observe that the columns of $\sigma_{\min}$ have the following structure (the notation $\widehat{m}$ indicates the absence of the element $m$ in a sequence).
\begin{itemize}
\item Every other odd special column starting with the second leftmost one is of the form $(j,j+1,\ldots,2i-2,2i-1)$.
\item Every other odd special column starting with the leftmost one is of the form: $(j,j+1,\ldots,2i-2,\widehat{2i-1},2i,2i+1,\ldots,2l)$ with $i\le l$, or $(j,j+1,\ldots,2l-1,2l,2i)$ with $l<i$.
\item A non-odd special column is of the form $(j,j+1,\ldots,2n-1,2n)$.
\end{itemize}
In particular, the second fact follows from the first two rules in Step~3 of Algorithm~\ref{algmin}. Based on these facts, we can describe each  bracketing for the column word of  $\sigma_{\min}$, which will prove (P1) and (P2).

Let us first bracket $(2i,2i-1)$ and ignore all such pairs coming from the same column of $\sigma_{\min}$. The remaining subword in these letters starts with a set of pairs $(2i,2i-1)$ coming from successive odd special columns, and ends with an even number of $2i$; the latter are all in the first row, with the one exception indicated in (P1) above, which corresponds to the number of odd special columns being odd (the entry $2i$ below the first row is in the rightmost odd special column). Here we applied Lemma~\ref{oddeven}~(3). Now let us bracket $(2i+1,2i)$ and again ignore all such pairs coming from the same column of $\sigma_{\min}$. The remaining subword in these letters consists of a sequence of $2i$ followed by a sequence of $2i+1$, where all the elements of the latter are in the first row. Indeed, no column can contain $2i+1$ below the first row but no $2i$ above it. 
\end{proof}


\section{Comparing the Sundaram and Kwon branching rules}\label{sunkwo}

The work in Sections~\ref{sundaram-lr} and \ref{genexpc} raises the question whether the Sundaram and Kwon branching rules (mentioned in those sections) are, in fact, equivalent. Based on the results above, we discuss what this equivalence entails, and we present an example which provides evidence for an affirmative answer. 

We consider the branching coefficient $c_{\nu}^{\lambda}(\mathfrak{sp}_{2n})$, for fixed $\lambda\in{\mathcal P}_n$ and $\nu\in{\mathcal P}_{2n}$. The Sundaram rule says that $c_{\nu}^{\lambda}(\mathfrak{sp}_{2n})$ is the number of Sundaram-LR tableaux of shape $\nu/\lambda$ and content $\delta$, for some $\delta\in{\mathcal P}_{2n}^{(1,1)}$. By Corollary~\ref{kwonbrrule} and Lemma~\ref{lemdelta}, the same coefficient is expressed as the number of LR tableaux $T$ in $LR_{\lambda,\delta}^\nu$ for which $S(T)$ satisfies the flag condition (C3) in Section~\ref{ranknlusz}, where $\delta\in{\mathcal P}_{2n}^{(1,1)}$ (recall the notation in Section~\ref{lrcs}). 

To relate the two types of tableaux, we need to consider the composition of the following maps:
\begin{equation}\label{compmaps}\{\mbox{LR tableaux of shape $\nu/\lambda$, content $\delta$}\}\:\xRightarrow{\rm{companion}}\:LR_{\delta\lambda}^\nu \:\xRightarrow{R\mbox{-}\rm{matrix}}\: LR_{\lambda,\delta}^\nu \:\xRightarrow{S}\: S(LR_{\lambda,\delta}^\nu)\,.\end{equation}
For the combinatorial $R$-matrix, we use the Henriques-Kamnitzer commutor \cite{hakccc,katccd}, which has several other realizations, cf. \cite{aktlr} and the references therein. Note that the Henriques-Kamnitzer commutor was defined in terms of the Sch\"utzenberger involution, which connects it to the last map in \eqref{compmaps}, namely the Sch\"utzenberger involution in the crystal $B_{2n}(\lambda)$.

The main question is whether the composition \eqref{compmaps} bijects the tableaux mentioned above, coming from the Sundaram and Kwon branching rules. The example below suggests an affirmative answer.

\begin{example}{\rm Consider $n=3$, $\lambda=(2,1,1)$, and $\nu=(5,4,3,3,3,2)$, with $c_{\nu}^{\lambda}(\mathfrak{sp}_{6})=1$. There are three LR tableaux of shape $\nu/\lambda$ for which the corresponding $\delta$ is in ${\mathcal P}_{2n}^{(1,1)}$. We indicate them below, together with the result of applying the maps in \eqref{compmaps}.

(1) $\delta=(3,3,3,3,2,2)$.

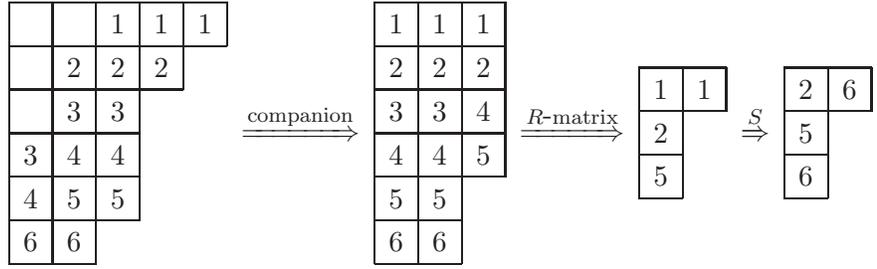
\[\tableau{{}&{}&1&1&1\\{}&2&2&2\\{}&3&3\\3&4&4\\4&5&5\\6&6}\:\xRightarrow{\rm{companion}}\:\tableau{1&1&1\\2&2&2\\3&3&4\\4&4&5\\5&5\\6&6}\:\xRightarrow{R\mbox{-}\rm{matrix}}\:\tableau{1&1\\2\\5}\:\xRightarrow{S}\:\tableau{2&6\\5\\6}\,.\]

(2) $\delta=(4,4,2,2,2,2)$.

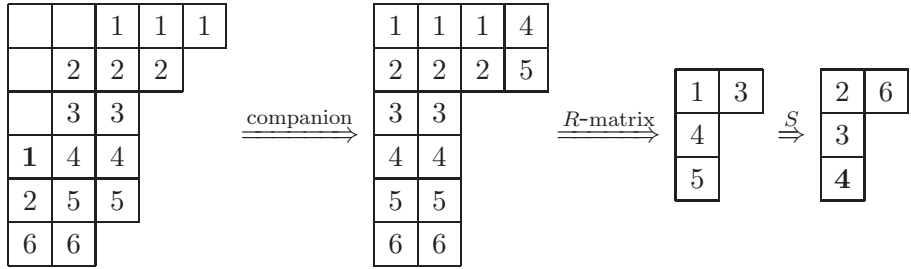
\[\tableau{{}&{}&1&1&1\\{}&2&2&2\\{}&3&3\\{\mathbf{1}}&4&4\\2&5&5\\6&6}\:\xRightarrow{\rm{companion}}\:\tableau{1&1&1&4\\2&2&2&5\\3&3\\4&4\\5&5\\6&6}\xRightarrow{R\mbox{-}\rm{matrix}}\:\tableau{1&3\\4\\5}\:\xRightarrow{S}\:\tableau{2&6\\3\\\mathbf{4}}\,.
\]

(3) $\delta=(4,4,3,3,1,1)$.

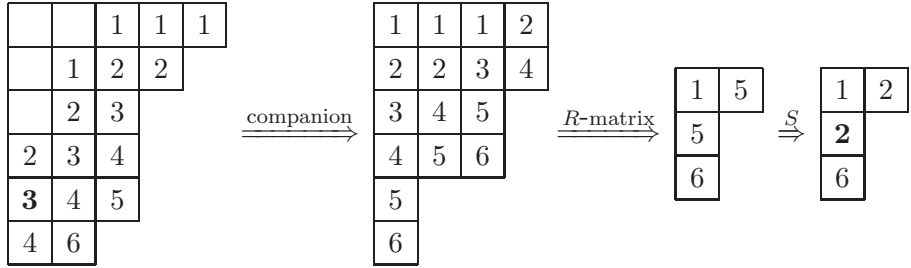
\[\tableau{{}&{}&1&1&1\\{}&1&2&2\\{}&2&3\\{2}&3&4\\{\mathbf{3}}&4&5\\4&6}\:\xRightarrow{\rm{companion}}\:\tableau{1&1&1&2\\2&2&3&4\\3&4&5\\4&5&6\\5\\6}\xRightarrow{R\mbox{-}\rm{matrix}}\:\tableau{1&5\\5\\6}\:\xRightarrow{S}\:\tableau{1&2\\\mathbf{2}\\6}\,.
\]

Note that in case (1) the first tableau is a Sundaram-LR tableau, while the last one satisfies condition (C3) mentioned above. However, both of these properties fail in cases (2) and (3); the entries causing these failures are shown in bold. 
}
\end{example} 

\begin{remark} {\rm Kwon's rule also works in orthogonal types, whereas there is no Sundaram-type rule in this case. For symplectic types, there is also the rule conjectured by Naito-Sagaki \cite{NS}, which was proved via its9 relation to the Sundaram rule in \cite{ST}.}\end{remark}

\end{document}